\documentclass[11pt]{amsart}

\usepackage{amsfonts, amssymb, amscd}

\usepackage{verbatim}
\usepackage{amssymb}
\usepackage{mathrsfs}
\usepackage{graphicx}
\usepackage{cite}
\usepackage[all]{xy}
\usepackage{tikz}
\usepackage[toc,page]{appendix}

\newcommand{\Zz}{\mathbb{Z}}
\newcommand{\Cc}{\mathbb{C}}
\newcommand{\Pp}{\mathbb{P}}
\newcommand{\Rr}{\mathbb{R}}
\newcommand{\Qq}{\mathbb{Q}}

\newcommand{\xx}{\mathbf{x}}

\newcommand{\Conv}{\operatorname{Conv}}

\newcommand{\Ann}{\operatorname{Ann}}

\newcommand{\Hom}{\operatorname{Hom}}
\newcommand{\rk}{\operatorname{rank}}

\newcommand{\la}{\langle}
\newcommand{\ra}{\rangle}

\newcommand{\Aa}{\mathcal{A}}
\newcommand{\Bb}{\mathcal{B}}

\newcommand{\Oo}{\mathcal{O}}

\newcommand{\Ee}{\mathcal{E}}
\newcommand{\Ff}{\mathcal{F}}
\newcommand{\Hh}{\mathcal{H}}
\newcommand{\Ll}{\mathcal{L}}

\newcommand{\Xx}{\mathcal{X}}
\newcommand{\Ss}{\mathcal{S}}
\newcommand{\Yy}{\mathcal{Y}}

\newcommand{\DT}{D((T_{i,j})_{i,j})}

\newtheorem{theorem}{Theorem}[section]
\newtheorem{lemma}[theorem]{Lemma}
\newtheorem{proposition}[theorem]{Proposition}
\newtheorem{definition}[theorem]{Definition}

\newtheorem{corollary}[theorem]{Corollary}
\newtheorem{remark}[theorem]{Remark}
\newtheorem{conjecture}[theorem]{Conjecture}

\numberwithin{equation}{section}

\begin{document}
\title{On Clifford  double mirrors of toric complete intersections}

\begin{abstract}
We present a construction of noncommutative double mirrors to
complete intersections in toric varieties. This construction unifies
existing sporadic examples and explains the underlying combinatorial
and physical reasons for their existence.
\end{abstract}

\author{Lev A. Borisov}
\address{Department of Mathematics\\
Rutgers University\\
Piscataway, NJ 08854} \email{borisov@math.rutgers.edu}

\author{Zhan Li}
\address{Beijing International Center for Mathematical Research\\
Beijing 100871, China} \email{lizhan@math.pku.edu.cn}

\maketitle

\tableofcontents

\section{Introduction.}\label{section1}
Two Calabi-Yau varieties $X$ and $Y$ are called a mirror symmetric pair if, together with some
K\"ahler data, they give rise to two superconformal field theories that differ by a twist, see \cite{CK99}. While this string-theoretic statement can not be at present rigorously understood, some
of its
consequences can be stated and even proved mathematically. For example, the (stringy)
Hodge numbers of the mirror pair $X$ and $Y$ are expected to obey
the relation $h^{p,q}(X) = h^{\dim Y -p, q}(Y)$. This so-called mirror duality test of an alleged mirror pair connects rather accessible invariants of $X$ and $Y$ and is the easiest to verify.
A significantly more complicated test of mirror symmetry connects quantum cohomology of $X$
with period integrals of $Y$. Even in the simple case when $X$ is a smooth quintic hypersurface
in $\Pp^4$, this is a highly nontrivial result due to Givental \cite{Givental}, which was later clarified
by Lian, Liu and Yau
in \cite{LLY}.
When one considers surfaces with boundary (open strings), then homological mirror symmetry \cite{Kon94} predicts that
the bounded derived category of coherent sheaves on $X$ is  equivalent to the Fukaya
category of $Y$ with the appropriate symplectic structure.

\medskip
An arbitrary Calabi-Yau variety may not always possess a mirror,
moreover, even if a mirror exists, it may not be unique. In fact, it is common for a Calabi-Yau
variety $Y$ to possess multiple mirror partners $X_i$. In physics terms the expectation
is that the superconformal field theories associated to $X_i$ are obtained from each other
by some kind of analytic continuation along the parameter space of such theories.
In this case, it is reasonable to refer to $X_1$ and $X_2$ as double mirrors of each
other in the sense that $X_2$ is a mirror of a mirror of $X_1$ and vice versa.
Even more generally, we will say that $X_1$ and $X_2$ are double mirror to each other if
they pass some basic compatibility tests below.\footnote{As such, we don't require that their mirror
exists in any sense.}

\begin{itemize}
\item
Hodge numbers of $X_1$ and $X_2$ coincide.

\item
Complex moduli spaces of $X_1$ and $X_2$ coincide.

\item
Quantum cohomology of $X_1$ and $X_2$ are obtained from each other by analytic continuation.

\item
Bounded derived categories of coherent sheaves on $X_1$ and $X_2$ coincide
(under some identification of the aforementioned complex moduli spaces).

\end{itemize}

Common examples of such $X_1$ and $X_2$ differ from each other by flops or more generally
by $K$-equivalences. In this case the last two statements are known as Ruan's and Kawamata's conjectures respectively.

\medskip
There are other prominent examples of double mirror Calabi-Yau varieties, such as
the Pfaffian-Grassmannian example, where $X_1$ and $X_2$ are not birational. In addition, it is natural to move a little bit beyond the category of algebraic varieties to allow Deligne-Mumford stacks, as well as mildly noncommutative ``varieties".
\footnote{While we do not wish to be explicit in this definition, by a mildly noncommutative
variety we mean a sheaf of finite rank algebras over a usual variety or stack.}
It is the latter kind of varieties that are the subject of this paper.

\medskip
There is a number of results in the literature where some usual or noncommutative varieties
$X_1$ and $X_2$ satisfy $D^b(Coh-X_1)=D^b(Coh-X_2)$ in the sense of equivalence of
categories. It is reasonable to postulate that most if not all such examples should be viewed as particular cases of the double mirror phenomenon.\footnote{As far as we know, there is no systematic study of nonlinear sigma models with noncommutative targets in the physics literature. Perhaps this paper may provide a motivation for it.} One such example is a construction
 of Kuznetsov, who shows that the derived category of a Calabi-Yau complete intersection
 of $k$ quadrics in $\Cc\Pp^{2k-1}$ is derived equivalent to a certain noncommutative crepant resolution of the double cover of $\Cc\Pp^{k-1}$ ramified over a determinant of a symmetric
 $2k\times 2k$ matrix of linear forms, twisted by a Brauer class. More precisely, this
 variety can be viewed as a certain sheaf of even Clifford algebras over $\Cc\Pp^{k-1}$.
A related example has been also considered by Calabrese and Thomas \cite{CT14}.

\medskip
The goal of this paper is to uncover the toric geometry that
underlies Kuznetsov's and Calabrese-Thomas'
constructions. This more general
framework allows us to construct additional examples and leads to
the more conceptual understanding of these derived equivalences.

\begin{remark}
We should also point out that there are examples of derived equivalences between noncommutative varieties which are not covered by our construction. For instance, C\u ald\u araru studied the derived category of elliptic fibration and the twisted derived category of its relative Jacobian \cite{Cal00a, Cal00b}. This is a family version of classical derived equivalence between abelian varieties. The twists are used to glue  the universal object in the Fourier-Mukai transform which may not exist in the ordinary sense. There are also Hosono-Takagi examples \cite{HT14},
which are closer in spirit to this paper, since they involve quadric fibrations. However, they
appear to be non-toric in nature and thus not covered by our construction.
\end{remark}

\medskip
Our construction starts off with a pair of dual reflexive Gorenstein cones $(K, K^\vee)$
in dual lattices $M$ and $N$. These are dual cones in the usual sense, with the property
that lattice generators of rays of $K$ and $K^\vee$ lie in the hyperplanes $\la -,\deg^\vee\ra=1$
and $\la \deg,- \ra=1$ respectively, where $\deg$ and $\deg^\vee$ are (uniquely defined)
lattice elements of $K$ and $K^\vee$. We consider
decompositions of the degree element $\deg^\vee$ in $K^\vee$ under certain appropriate conditions. We associate noncommutative varieties to such decompositions, and the change of the decomposition   results in conjectural double mirrors.

\medskip
To make this a bit more precise, suppose we have
\begin{equation}\label{decintro}
\deg^\vee = \frac 1 2 (s_1 + \cdots + s_{2r}) + t_1 + \cdots + t_{k-r}
\end{equation}
where $s_i$ and $t_j$ are linearly independent lattice elements of $K^\vee$
which satisfy $\la \deg,s_i\ra =\la \deg,t_j\ra=1$. The index $k=\la \deg,\deg^\vee\ra$
 of reflexive Gorenstein pair $(K,K^\vee)$ is fixed, but $r$ could vary.
We construct a stack $\Ss$ which is a complete intersection in a
toric stack by equations associated to $\{t_i\}$.
Then $\{s_i\}$ give a sheaf of even Clifford
algebras $\Bb_0$, and they combine to give us a
noncommutative variety $(\Ss,\Bb_0)$. \footnote{For our construction
to work best, we need some additional technical centrality condition
on a fan $\Sigma$ in $K^\vee$ and a certain
flatness assumption.}

\medskip
Two extreme cases are of particular importance: when $r=0$, there
are no {$\{s_i\}$} in the expression of $\deg^\vee$
and $\Bb_0 = \Oo$. The noncommutative stacks $(\Ss, \Bb_0)$ are just
usual DM stacks $\Ss$ which are crepant resolutions of Calabi-Yau
complete intersections in toric Gorenstein Fano varieties. If there
is a mirror, then one gets the classical Batyrev-Borisov
construction. On the other end of the spectrum, when $r=k$, there
are no {$\{t_i\}$} in the expression of $\deg^\vee$.
In this case $\Ss$ is a toric DM stack.

\medskip
The main result of the paper is the following theorem.

\medskip
\noindent
\begin{theorem}{\bf (= Theorem \ref{dmthm})}
Suppose that a complete intersection $\Xx$ and a Clifford noncommutative variety $\Yy$
are given by different decompositions of the degree element $\deg^\vee$ of a reflexive Gorenstein
cone $K^\vee$ and the appropriate regular simplicial fans in $K^\vee$. Then the bounded derived categories of $\Xx$ and $\Yy$ are equivalent, provided the centrality and the flatness assumptions on $\Yy$ hold.
\end{theorem}

\medskip
\medskip
For the benefit of the reader, we try to keep the paper as self-contained as possible.
It is structured as follows.

\medskip
In Section \ref{section2} we review the definition of reflexive
Gorenstein cones and set up some of the {notations}
that recur throughout the paper. We give a quick
introduction to Batyrev-Borisov construction from the viewpoint of
the pairs of reflexive Gorenstein cones $(K,K^\vee)$. This is a
slightly different approach to the subject than the more traditional
study of nef-partitions as in \cite{Bor93}. We find it more natural both in its own
right and for the purposes of this paper. In fact, we only mention
nef-partitions in passing remarks.

\medskip
In Section \ref{philosophy} we outline the physical intuition that guides this paper and is key to
the proper understanding of the construction and its possible generalizations. We hope that even readers not interested in the more technical details of the rest of the paper will read this section.
We argue that the principal category of interest is the graded equivariant derived category of factorizations of  the potential in the homogeneous coordinates, considered in \cite{BFK12}.
It is defined for every regular simplicial fan in $K^\vee$ but is unique up to equivalence.

\medskip
In Section \ref{K.recall} we recall the construction of Kuznetsov of Clifford double mirrors of
complete intersections of quadrics in  $\Cc\Pp^n$. This section serves a dual purpose. On the one
hand, we introduce the key example that served as the original motivation behind this paper.
On the other hand, we introduce (even) Clifford algebras which will be used further in the paper.

\medskip
In Section \ref{defcliff} we consider the decompositions  \eqref{decintro} with $r=k$. We construct
sheaves of Clifford algebras over toric DM stacks which generalize Kuznetsov's construction and open the door to many more examples of the phenomenon. In particular, we describe a double mirror to the quotient of complete intersection of four quadrics in $\Cc\Pp^7$ by a fixed point free involution.

\medskip
In Section \ref{sectionderived} we prove the first case of our main result, namely the equivalence of derived categories of Clifford double mirrors to the graded equivariant derived categories of factorizations from Section \ref{philosophy}. Our argument is based heavily on
the work of \cite{BFK12}.

\medskip
In Section \ref{general} we generalize the construction of Section \ref{defcliff} to the $0<r<k$
case of \eqref{decintro}. We also conjecture that the derived equivalence statement of Section \ref{sectionderived} extends to this more general case.

\medskip
In Section \ref{sectioncomb} we discuss in detail the combinatorics of Clifford double mirrors.
It is likely to be useful in further study of the phenomenon.

\medskip
In Section \ref{sectionex} we describe several additional examples
of the construction. Some of them such as Calabrese-Thomas'
example \cite{CT14}  already appear in
the literature, and the others are new. We specifically look at what
happens if some of the assumptions of the main theorem are relaxed.

\medskip
Finally, in Section \ref{sectionrem} we make some concluding remarks and pose open questions
that we hope the readers or the authors will address in future research.

\medskip
\noindent{\it Acknowledgements.} \ We thank Nicolas Addington, Matt Ballard, Tyler Kelly, Alexander Kuznetsov and Howard Nuer for stimulating conversations and useful comments.
L. Borisov was partially supported by NSF grant DMS-1201466.

\section{Review of reflexive Gorenstein cones, Batyrev-Borisov mirror construction and double mirror phenomenon.}\label{section2}
In this section we give an overview of Batyrev-Borisov mirror construction with the emphasis on reflexive Gorenstein
cones. This viewpoint is essential for understanding the rest of the paper. So even a reader familiar with Batyrev-Borisov
construction in terms of nef-partitions will find it necessary to at least briefly look through this section. We describe the
crucial work of Batyrev and Nill \cite{BN08} which forms the basis for the double mirror phenomenon in the Batyrev-Borisov setting.
In the process we fix the notations and explain the way they are used throughout the paper.

\subsection{Reflexive Gorenstein cones.}\label{subsection: Reflexive Gorenstein cones}

Let $M \cong \Zz^{\rk M}$ be a lattice and let $N:=\Hom_\Zz(M, \Zz)$ be its dual. The natural pairing is given by
$$\la , \ra : M \times N \to \Zz.$$
Let $M_\Rr:= M \otimes_\Zz \Rr, N_\Rr:= N \otimes_\Zz \Rr$ be the $\Rr$-linear extensions of $M, N$. The pairing can be $\Rr$-linearly extended as well, and we still use $\la , \ra$ to denote this extension.

\begin{definition}
A rational polyhedral cone $K \subset M_\Rr$ is a convex cone generated by a finite set of lattice points. We assume that $K\cap (-K)=\{0\}$.
We call the first lattice point of a ray $\rho$  of $K$ a primitive element or a lattice generator of $\rho$.
\end{definition}

\begin{definition}(\cite{BB94})
A full-dimensional rational polyhedral cone $K \subset M_\Rr$ is called a \emph{Gorenstein} cone if all the primitive elements of its rays lie on some hyperplane $\la - , n \ra =1$ for some \emph{degree element} $n$ in $N$.
A Gorenstein cone $K \subset M_\Rr$ is called a \emph{reflexive} Gorenstein cone iff its dual cone $K^\vee: =\{y \mid \la x, y \ra
\geq 0 ~\forall~ x \in K\}$ is also a Gorenstein cone with respect to the dual lattice $N$.
\end{definition}

\begin{remark}
Note that $(K^\vee)^\vee=K$, which is why we typically talk about a pair of dual reflexive
Gorenstein cones $K$ and $K^\vee$. For any such pair
the degree elements are unique and are denoted by $\deg^\vee$ and $\deg$ respectively.
Observe that $\deg^\vee$ is a lattice point in the interior $(K^\vee)^\circ$ of
$K^\vee$ and there holds
$$
(K^\vee)^\circ\cap N = \deg^\vee+ (K^\vee\cap N).
$$
Similarly $K^\circ\cap M=\deg+(K\cap M)$.
\end{remark}

\begin{definition}
For a pair of dual reflexive Gorenstein cones $(K,K^\vee)$, the pairing of their two degree elements
$\la \deg, \deg^\vee \ra = k$ is called the \emph{index} of the pair. The index is
always a positive integer.
\end{definition}

The following is an example of a $3$-dimensional reflexive Gorenstein cone and its dual cone. Notice that two degree elements $\deg, \deg^\vee$ happen to lie on  the hyperplanes $\la - , \deg^\vee \ra =1$ and $\la \deg, -\ra=1$ which may not be the case in general. Indeed, this will never happen as soon as the index of the pair of Gorenstein cones is larger than $1$.
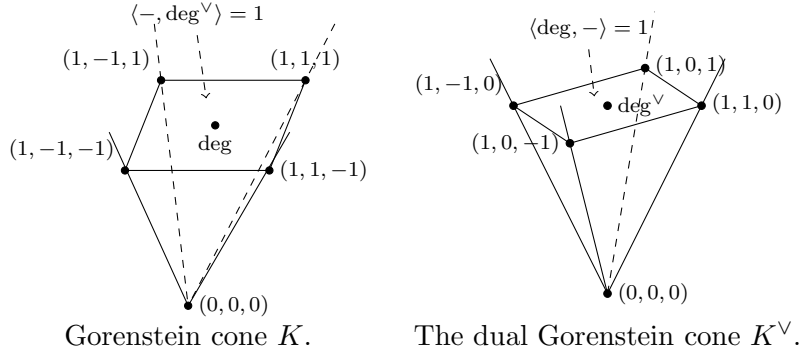
\[\begin{array}{cc}
\begin{tikzpicture}[scale=1.2]
\draw (0,1) -- (1.6, 1);
\draw (-0.4,0) -- (1.2, 0);
\draw (0, 1) -- (-0.4,0);
\draw (1.6, 1) -- (1.2, 0);
\draw[dashed] (-0.075, 1.625) -- (0.3,-1.5);
\draw[dashed] (1.925, 1.625) -- (0.3,-1.5);
\draw (-0.575, 0.425) -- (0.3,-1.5);
\draw (1.425,0.425) --(0.3,-1.5);

\draw[fill] (.6,0.5) circle (.04);
\node[below] at (.6,0.5) {\scriptsize $\deg$};

\draw[fill] (1.2, 0)  circle (.04);
\node[right] at (1.2, 0) {\scriptsize$ (1,1,-1)$};
\draw[fill] (1.6, 1)  circle (.04);
\node[above] at (1.6, 1) {\scriptsize$(1,1,1)$};
\draw[fill] (0, 1)  circle (.04);
\node[above left] at (0, 1) {\scriptsize$(1,-1,1)$};
\draw[fill] (-0.4, 0)  circle (.04);
\node[above left] at (-0.4, 0) {\scriptsize$(1,-1,-1)$};
\draw[fill] (0.3,-1.5)  circle (0.04);
\node[right] at (0.3,-1.5) {\scriptsize$(0,0,0)$};
\node[below] at (0.3,-1.6) {Gorenstein cone $K$.};

\draw [->, dashed]  (.4, 1.5) -- (.5,.8);
\node[above] at (.4, 1.5) {\scriptsize$\la -, \deg^\vee \ra =1$};
\end{tikzpicture}

&

\begin{tikzpicture}[scale=2.5]
\draw (0.5,1) -- (.2, 1.2);
\draw (.2, 1.2) -- (-.5,1);
\draw (-.5,1) -- (-0.2,.8);
\draw (-0.2,.8) -- (.5, 1);

\draw (0,0) -- (.625, 1.25);
\draw[dashed] (0,0) -- (.25, 1.5);
\draw (0,0) -- (-.625, 1.25);
\draw (0,0) -- (-.25, 1);

\draw[fill] (0,1) circle (0.02);
\node[right] at (0,1) {\scriptsize$\deg^\vee$};

\draw[fill] (0.5,1) circle (0.02);
\node[above right] at (0.5,.9) {\scriptsize$(1,1,0)$};
\draw[fill] (0.2,1.2) circle (0.02);
\node[above right] at (0.2,1.1) {\scriptsize$(1,0,1)$};
\draw[fill] (-0.5,1) circle (0.02);
\node[above left] at (-0.5,1) {\scriptsize$(1,-1,0)$};
\draw[fill] (-.2,.8) circle (0.02);
\node[left] at (-.2,.8) {\scriptsize$(1,0,-1)$};
\draw[fill] (0,0) circle (0.02);
\node[right] at (0,0) {\scriptsize$(0,0,0)$};
\node[below] at (0.0,-.1) {The dual Gorenstein cone $K^\vee$.};

\draw [->, dashed]  (-.1, 1.3) -- (-0.07, 1.05);
\node[above] at (-.1, 1.3) {\scriptsize$\la \deg, - \ra =1$};
\end{tikzpicture}
\end{array}\] We will revisit this example in Section \ref{subsection: a first non-trivial example}.

\begin{remark}
If a pair of Gorenstein cones leads to a complete intersection in a toric variety, the index $k$
of the pair is the codimension of the complete intersection. The dimension of the complete intersection is
$d-2k$ where $d={\rk M}={\rk N}$. In particular, the original Batyrev's hypersurface construction corresponds to
$k=1$, as is the case in the above figure. In this paper we are primarily interested in the case $k>1$, so the reader
should not rely too much on their knowledge of the original Batyrev's hypersurface construction.
\end{remark}

Given a pair of reflexive Gorenstein  cones $(K, K^\vee)$ we define two lattice polytopes
$$
\Delta =\{x\in K, \la x,\deg^\vee\ra = 1\},~~
\Delta^\vee =\{y\in K^\vee, \la \deg, y\ra = 1\}.
$$
Their sets of lattice points are denoted by
$$
K_{(1)}:=\Delta\cap M,~~K^\vee_{(1)}:=\Delta^\vee\cap N
$$
respectively.
When the index $k$ is one, these are the original reflexive polytopes of Batyrev \cite{Bat94}. For $k>1$, these polytopes
have no interior lattice points, although they can and often do have non-vertex lattice points on the boundary.

\medskip
A crucial part of the data necessary to define Calabi-Yau varieties (or more generally triangulated
categories of Calabi-Yau type) in the setting of Gorenstein reflexive cones, is a family of \emph{coefficient functions}
$$
c:K_{(1)}\to \Cc.
$$
We typically fix an element of this family in general position.

\subsection{Reflexive cones to complete intersections.}
The original Borisov's extension of Batyrev's construction was accomplished by the use of nef-partitions.
However, a more flexible and conceptually superior way of looking at the construction has been later provided
by the work of Batyrev and Nill \cite{BN08}. This new approach allows for a very clear way of constructing double mirrors
in the Batyrev-Borisov setting. While we follow the idea of Batyrev and Nill, the exposition below is different
from their paper. Indeed, we are trying to set up the viewpoint that will naturally extend to our noncommutative
double mirror setting.

\medskip
The main idea  of Batyrev-Nill's paper is the following concept of the decomposition of the dual degree element.
\begin{definition}
Let $(K\subset M_\Rr,K^\vee\subset N_\Rr)$ be a pair of dual reflexive Gorenstein cones of index $k$.
We call
$$
\deg^\vee =t_1+\cdots + t_k
$$
a decomposition of $\deg^\vee$ if all $t_i$ are elements  of $K^\vee_{(1)}$.
\footnote{It is easy to show that these $t_i$ are linearly independent, so it does not need to be a part of the definition.}
\end{definition}

Given a decomposition of $\deg^\vee$, one can construct a singular toric variety $\mathbb P_{sing}$ and
a  family of Calabi-Yau complete intersections in it. We will later describe how one can construct
a Deligne-Mumford stack crepant resolution  of this family.
\begin{definition}
Since $t_i\in K^\vee\cap N$, the pairing with $t_i$ provides the semigroup ring
$\Cc[K\cap M]$ with a $(\Zz_{\geq 0})^k$ grading. Indeed, a monomial associated to $m$
will have grading $(\la m,t_1\ra,\ldots,\la m,t_k\ra)$. One can then define the multi-Proj of this ring
the same way one defines a usual ${\rm Proj}$ to get
$$
\mathbb P_{sing}:= {\rm multiProj}(\Cc[K\cap M]).
$$
\end{definition}

\begin{remark}
We will later see a more conventional definition of this toric variety $\mathbb P_{sing}$, which does not use the multi-Proj construction but is a bit less intuitive.
\end{remark}

The decomposition $\deg^\vee = \sum_{i=1}^k t_i$ provides a decomposition of the set of lattice points of $\Delta$ into
a disjoint union of the sets
$$
K_{(1)} =\bigsqcup_{i=1}^k K_{(1),i},~K_{(1),i} = \{x\in K_{(1),i},
\la x, t_j\ra = \delta_{i,j} \}
$$
where $\delta_{i,j}$ is the Kronecker delta. Importantly,
these sets $K_{(1),i}$ are the sets of lattice points of lattice
polytopes $\Delta_i$ in $M$, which are faces of $\Delta$. A generic
coefficient function $c:K_{(1)}\to \Cc$ now allows one to define $k$
homogeneous elements of $\Cc[K\cap M]$
$$
\sum_{m\in K_{(1),1}} c(m) [m], \ldots, \sum_{m\in K_{(1),k}} c(m) [m],
$$
where $[m]$ is the monomial element of $\Cc[K\cap M]$ that corresponds to $m$.

\begin{definition}
For a generic choice of $c$, we define the complete intersection $X_{c,sing}\subset\mathbb P_{sing}$ by
$$
X_{c,sing} := {\rm multiProj}(\Cc[K\cap M]\slash\la \sum_{m\in K_{(1),1}} c(m) [m], \ldots, \sum_{m\in K_{(1),k}} c(m) [m]\ra).
$$
\end{definition}

\begin{remark}
In the absence of a decomposition of $\deg^\vee$, one may only consider the usual
$$
H={\rm Proj}(\Cc[K\cap M]\slash\la \sum_{m\in\Delta} c(m) [m]\ra)\subset
{\rm Proj}(\Cc[K\cap M])
$$
which is a hypersurface in a toric variety. A choice of a decomposition allows one to realize the above
hypersurface as a so-called Cayley hypersurface of a complete intersection. There are also situations where
a decomposition of $\deg^\vee$ does not exist, as in the example of \cite{BB94}. In these cases one
can try to work with the hypersurface $H$ as if it were a Cayley hypersurface of a complete intersection, even though no such complete intersection is available.
\end{remark}

Let us now describe $\mathbb P_{sing}$ in a more traditional way, which will also allow us to talk
about its desingularizations.

\medskip
Observe that the sublattice of $N$ given by $\Zz t_1+\ldots +\Zz t_k$ is
of rank $k$ and is saturated. Indeed, if a set $K_{(1),i}$ were empty for some $i$, it would mean that
all elements of $K_{(1)}$ had zero pairing with $t_i$, in contradiction to $K$ being full-dimensional. Thus, there are
elements in each $K_{(1),i}$ and a pairing with them provides a splitting to the natural map $\Zz^k \to N$ given
by $t_i$. We define the quotient lattice $\overline N$ of $N$ and a sublattice $\overline M$ of $M$ by
$$
\overline N := N/(\Zz t_1+\ldots +\Zz t_k),~~\overline M:={\rm Ann}(\Zz t_1+\ldots +\Zz t_k).
$$
Note that
$\overline M$ and $\overline N$ are naturally dual to each other.
The image of the polytope $\Delta^\vee$ under the quotient map
$\phi:N_\Rr\to \overline N_\Rr$ is a polytope $\phi(\Delta^\vee)$
which is reflexive (see \cite{Li13}). Consider the minimum fan
$\Sigma_1$ in $\overline N_\Rr$ associated to $\phi(\Delta^\vee)$,
i.e. the fan whose maximum cones correspond to facets of
$\phi(\Delta^\vee)$. Then $\mathbb P_{sing}$ is the toric Fano
variety associated to $\Sigma_1$.

\medskip
Note that the polytopes $\Delta_i$ lie in parallel translates of $\overline M$. We also observe \cite{Li13} that
$$
\sum_{i=1}^k\Delta_i-\deg
$$
is the reflexive polytope dual to $\phi(\Delta^\vee)$.
This allows us to view polytopes $\Delta_i$ as support polytopes for global sections of $k$ globally generated line
bundles on $\mathbb P_{sing}$. The Minkowski sum of the polytopes is the anti-canonical polytope, which
means that the tensor product of the line bundles is the anti-canonical bundle of $\Pp_{sing}$.
Thus, provided that the intersection $X_{c,sing}$ is of the expected dimension, it will be a Calabi-Yau variety
by the adjunction formula. For a generic choice of the coefficient function $c$ the
resulting $X_{c,sing}$ is of the correct dimension and has a DM stack resolution induced by a resolution
of the ambient space $\Pp_{sing}$.

\medskip
Before we prove the main result of this section, let us state and prove a simple lemma.
\begin{lemma}\label{above}
We have
$$
K= \sum_i \Rr_{\geq 0} \Delta_i.
$$
Moreover, if a point $v\in K$ has $\la v, t_i\ra = \alpha_i$, then
$v\in \sum_i \alpha_i \Delta_i$.
\end{lemma}

\begin{proof}
Since $\Delta_i\subset K$, we have $\sum_i \Rr_{\geq 0} \Delta_i\subseteq K$.
In the other direction, observe that for every ray of $K$ its generator lies in one of the $\Delta_i$.
Finally, the last statement follows by considering the pairing with $t_1,\ldots,t_k$.
\end{proof}

Let us now formulate an important result that connects triangulations of the boundary of $\phi(\Delta^\vee)$ with fans on $K^\vee$.
\begin{proposition}\label{2.12}
Let $\overline \Sigma$ be a simplicial fan in $\overline N$ which comes from a regular triangulation of the boundary of $\phi(\Delta^\vee)$.
Consider the fan $\Sigma$ in $N$ which is the preimage of $\overline\Sigma$ intersected with $K^\vee$. Then preimages of the maximum cones of $\overline\Sigma$ are themselves simplicial cones in $N_\Rr$, and the fan $\Sigma$ comes from a regular
triangulation of $\Delta^\vee$.
\end{proposition}

\begin{proof}
It is clear that the preimages of the cones of $\overline\Sigma$ form a fan, when intersected with $K^\vee$.
To make sure that the generators of the rays of $\Sigma$ are lattice points of $\Delta^\vee$,
we need to show that all lattice points of $\phi(\Delta^\vee)$ are images of lattice points of $\Delta^\vee$.
Suppose that there is $p\in \phi(\Delta^\vee)\cap \overline N$. We have $p=\phi(q)$ for some $q\in \Delta^\vee$. While we can not assume that $q\in N$, we know that
there exist $\alpha_i\in\Rr$ such that
$$
\hat q=\sum_{i=1}^k \alpha_i t_i + q  \in  N.
$$
We may assume that $\hat q\in K^\vee$ and by picking such $\hat q$
of minimum degree $\la \deg,\hat q\ra$
 we can assure that for all $i$  we have $\hat q-
t_i\not\in K^\vee$. This means that there are ray generators of $K$ which we denote
by $v_i$
such that $\la v_i ,\hat q - t_i\ra<0$. Since $\la v_i ,\hat
q\ra \geq 0$, this shows that $\la v_i ,t_i\ra>0$, thus $v_i\in
\Delta_i\cap M$. Since $\la v_i,t_i\ra=1$ and $\la v_i ,\hat q\ra$
is integer, we see that $\la v_i ,\hat q\ra=0$. Since $ \la v_i
,\hat q\ra = \alpha_i + \la v_i , q\ra $ and $\la v_i,q\ra$ are
nonnegative, we see that $\alpha_i\leq 0$. Unless all $\alpha_i$ are
zero, we see that
$$
\la \deg,\hat q\ra = \sum_i \alpha_i +\la \deg,q\ra = \sum_i \alpha_i +1 <1.
$$
Since $\hat q\in N\cap K^\vee$, we get
$\hat q={\bf 0}$, in which case the above argument shows that the only lattice
preimages $q$ of ${\bf 0}\in \overline N$ inside $\Delta^\vee$ are $t_i$.
If all $\alpha_i$ are zero, it means that $q=\hat q$ is a lattice
point. In fact, the above argument show the uniqueness of such $q$.

\medskip
Let us now show that the intersection of the  preimage of a maximum-dimensional cone
$$
\sigma_1 = \Rr_{\geq 0}\phi(w_1)+\cdots+\Rr_{\geq 0}\phi(w_{d-k})
$$
of $\overline\Sigma$ with $K^\vee$ is the cone
$$
\sigma =  \Rr_{\geq 0}w_1+\cdots+\Rr_{\geq 0}w_{d-k}+\sum_{i=1}^k \Rr_{\geq 0}t_i.
$$
The inclusion
$$
\sigma \subseteq \phi^{-1}(\sigma_1)\cap K^\vee
$$
is clear. In the other direction, observe that there is a facet of $\phi(\Delta^\vee)$ that
contains all $\phi(w_j)$. This means that there is an element $v\in \overline M_\Rr\subset M_\Rr$ (we can pick it in $M$, but we will
not need it) such that $\la v,t_i \ra=0$ for all $i$,  $\la v, \Delta^\vee\ra \geq -1$ and $\la v, w_j\ra=-1$ for all $j$.
We see that $(v+\deg)$ is nonnegative on $\Delta^\vee$, therefore $v+\deg\in K$. By Lemma \ref{above}
we see that
$$
v+\deg = \sum_{i=1}^k v_k
$$
with $v_i\in \Delta_i$. Observe that $\la v+\deg , w_j \ra =0$ and $v+\deg \in K$ implies
that $\la v_i,w_j\ra = 0$ for all $i$ and $j$.

\medskip
Now suppose that
$$w = \sum_{i} \alpha_i t_i + \sum_j \beta_j w_j\in  \phi^{-1}(\sigma_1)\cap K^\vee$$
Then for all $i$ we have $\la v_i,w\ra \geq 0$ which  implies  that $\alpha_i\geq 0$, so $w\in \sigma$.

\medskip
Last, we observe that the fan $\Sigma$ is regular, since we may simply use the pullback of the
convex  piecewise-linear function on $\overline\Sigma$ to give one for $\Sigma$.
\end{proof}

By Proposition \ref{2.12}, we can construct a regular simplicial fan $\Sigma$ on $K^\vee$ with the
following property.
\begin{equation}\label{central}
\emph{All maximum dimensional cones of $\Sigma$ contain $t_1,\ldots,t_k$.}
\end{equation}

\medskip
The toric variety $\Pp_\Sigma$ is a vector bundle
over the toric variety $\Pp_{\overline\Sigma}$. However, we are interested in the corresponding Deligne-Mumford
stacks. We will review the construction  of these stacks briefly. It is much easier in this setting than
in general \cite{BCS05} since the lattices in question do not have torsion.

\medskip
Consider the open torus-invariant subset $U_\Sigma$ of the variety $\Cc^{K^\vee_{(1)}}$ of complex-valued functions
on the finite set $K^\vee_{(1)}$
given by the points ${\bf z}:K^\vee_{(1)}\to \Cc$ with the zero locus a subset $\sigma\in \Sigma$.
\footnote{Here we mean that the indices of the zero locus form a set in the simplicial complex that
corresponds to $\Sigma$. In particular, if a point in $K^\vee_{(1)}$ is not used in $\Sigma$, the corresponding
variable is always nonzero. The idea of using such invisible rays has appeared soon after the initial construction of stacky fans,
see for example \cite{J08}.}
There is a group $G$ which acts on $U_\Sigma$. It is the subgroup of the torus $(\Cc^*)^{K^\vee_{(1)}}$
which consists of ${\bf \lambda}:K^\vee_{(1)}\to \Cc^*$ with the property that
$$
\prod_{n\in K^\vee_{(1)}} \lambda(n)^{\la m,n\ra} = 1
$$
for all $m\in M$. The smooth Deligne-Mumford stack $[U_\Sigma/G]$ is
a resolution of singularities of ${\rm Spec}(\Cc[K\cap M])$.

\medskip
Similarly, we consider the open subset
$U_{\overline \Sigma}$ of the variety $\Cc^{K^\vee_{(1)}-\{t_1,\ldots,t_k\}}$ given
by the property that the zero locus of $\bf z$ lies in a cone of $\overline\Sigma$. We see that
$$
U_\Sigma = U_{\overline\Sigma}\times \Cc^k
$$
where the second factor corresponds to the values of ${\bf z}$ on $t_1,\ldots,t_k$. Moreover, our description of
the set of lattice points of $\phi(\Delta^\vee)$ implies that the group $\overline G$ used to construct the Deligne-Mumford
stack that corresponds to $\overline\Sigma$ coincides with the group $G$ above. Specifically, every function
$$
\overline \lambda:{K^\vee_{(1)}-\{t_1,\ldots,t_k\}}\to \Cc^*
$$
which satisfies
$$\prod_{n\in K^\vee_{(1)} - \{t_1, \ldots, t_k\}}
\overline \lambda(n)^{\la m,n\ra} = 1$$
for all $m\in \overline M =
{\rm Ann}(\Zz t_1+\cdots \Zz t_k)$ can be uniquely extended to
$$
 \lambda:K^\vee_{(1)}\to \Cc^*
$$
which satisfies the above condition for all $m\in M$. This gives $U_\Sigma$ a structure of the $G$-equivariant vector
bundle over $U_{\overline\Sigma}$ and gives $\Pp_\Sigma = [U_\Sigma/G]$ a structure of a vector bundle over
$\Pp_{\overline\Sigma} = [U_{\overline\Sigma}/G]$. Moreover, this vector bundle is naturally the direct sum of $k$ line bundles that
correspond to the individual coordinates $z(t_i)$.

\medskip
As before, consider a generic coefficient function $c$. It allows us to define polynomials $\overline h_i$ in coordinates
of $\Cc^{K^\vee_{(1)}-\{t_1,\ldots,t_k\}}$
\[
 \overline h_i:= \sum_{m\in K_{(1),i}} c(m) \prod_{n\in K^\vee_{(1)}-\{t_1,\ldots,t_k\}}
{z(n)^{\la m, n\ra}},
\]
and the corresponding complete intersection on $U_{\overline\Sigma}$
\[
\overline X_{c} := \{ \overline h_1 = \cdots =  \overline h_k =0\} \subset U_{\overline\Sigma}.
\]
Observe that we have
$$
C({\bf z})=\sum_{i=1}^k z(t_i)\overline h_i = \sum_{m\in K_{(1)}}
c(m) \prod_{n\in K^\vee_{(1)}} {z(n)^{\la m, n\ra}},
$$
and therefore $\{C=0\}$ is the Cayley hypersurface associated to the complete intersection of $\{\overline h_i=0\}$.
Since the above construction  is $G$-equivariant, we see that the same is true for the stacks
$$
\Hh_c:=[\{C({\bf z})=0\}/G]\subset \Pp_\Sigma
$$
and
$$
\Xx_c:=[\cap_i \{\overline h_i=0\}/G]\subset \Pp_{\overline\Sigma}.
$$
We observe that if $c$ is generic (for example $\Delta$-nondegenerate in the sense of Batyrev) then
$\Xx_c$ is a smooth Deligne-Mumford stack. The stack $\Hh_c$ is singular along $\Xx_c$ which is naturally
included as a substack in the zero section of the vector bundle $\Pp_{\Sigma}\to \Pp_{\overline\Sigma}$.

\medskip
We have thus described how to  associate to a decomposition
$$
\deg^\vee= t_1+\cdots +t_k
$$
and a fan $\Sigma$ with the property \eqref{central} a family of smooth Calabi-Yau DM stacks.

\subsection{Mirrors and double mirrors.}
To construct the mirror family, one should consider a decomposition
$$
\deg = u_1+\ldots + u_k
$$
where $u_i\in K_{(1)}$. One can reindex $u_i$ to ensure that $\la u_i, t_j \ra=\delta_{i,j}$. One can show that
the combinatorial data of $(K,K^\vee)$ together with the pair of decompositions of $\deg$ and $\deg^\vee$
are precisely equivalent to the data of the nef-partition considered originally in \cite{BB94}.
It allows one to construct  a family
$$
\{\Yy_{c^\vee}\}
$$
for generic mirror coefficient functions $c^\vee:K^\vee_{(1)}\to \Cc$, provided one picks a simplicial subdivision
of $K$.

\medskip
The double mirrors  of $\{\Xx_c\}$ are then defined as \emph{mirrors of mirrors}  of  $\{\Xx_c\}$. As the above
discussion shows, they can be obtained from the same pair of cones $(K,K^\vee)$ and the coordinate functions $c$
by changing the fan $\Sigma$ and more interestingly \emph{by changing a decomposition of $\deg^\vee$}
to
$$
\deg^\vee = t_1' +\cdots + t_k'.
$$
A striking observation regarding this double mirror construction is
that the coefficients $c(m)$ are generally sorted into different subsets
to define the complete intersection! Nonetheless the resulting stacks are expected to share many properties.

\begin{remark}
In what follows we will call $\{\Xx_{c}'\}$ a double mirror of
$\{\Xx_{c}\}$ even in the absence of a choice of a decomposition of
$\deg$. \footnote{Such decomposition may not exist, in which case
the families in questions are in some sense generalized
double-mirrors of each other. We do not expect this to be
essential.} We are less interested in the choice of $\Sigma$ and
$\Sigma'$, although they of course are needed. The reason is that
different choices of the fans amount to a composition of toric
flops.
\end{remark}

\medskip
We collect existing results related to Batyrev-Borisov double mirrors $\Xx, \Xx'$.

\begin{theorem}
For any $p, q$, the stringy Hodge numbers $h^{p,q}_{\rm st}(X_{c,sing})$ and $h^{p,q}_{\rm st}(X_{c,sing}')$ coinside.
\end{theorem}

This is a direct consequence of the main theorem of \cite{BB96} which provides a combinatorial
formula for $h_{\rm st}^{p,q}(X_{c,sing})$ in terms of the combinatorics of the cones $K$ and $K^\vee$. Note
that these Hodge numbers are the orbifold Hodge numbers of the crepant stacky resolutions $\Xx_{c}$
and $\Xx_c'$ considered above.

\medskip

In \cite{BN08}, Batyrev and Nill proposed conjectures on the birationality and derived equivalence of double mirrors. These conjectures have been answered in \cite{Li13, FK14}.

\begin{theorem}(\cite{Li13})
Under some mild technical assumptions, the double mirrors $\Xx_c$ and
$\Xx'_c$ are birational.
\end{theorem}

\begin{theorem}(\cite{FK14})
The double mirrors $\Xx_c$ and $\Xx'_c$ are derived equivalent.
\end{theorem}

There are other types of (commutative) mirror constructions in mirror symmetry which also exhibit double mirror phenomenon. Among them,  the so called Berglund-H\"ubsch-Krawitz mirror construction  \cite{BH93,Kra09} is particularly well understood.  We briefly mention the parallel results in BHK setting.

\medskip

Let $Z_{A,G}, Z_{A',G}$ be a pair of
BHK double mirrors. The equivalence between orbifold Chen-Ruan cohomology of corresponding DM-stacks $[Z_{A,G}]$ and  $[Z_{A',G}]$ is a consequence of the main result of  \cite{CR11}.  The birationality of double mirrors
is established in various generality in \cite{Sho12, Kel13, Cla13, Bor13}. In \cite{FK14}, the derived categories of double mirrors are shown to be equivalent.

\section{The underlying philosophy: triangulated categories associated to reflexive Gorenstein cones.}\label{philosophy}
In this section we explain the underlying philosophy that guides our construction. We construct triangulated categories
of type IIB boundary conditions for the data of reflexive Gorenstein cones and corresponding potentials. We argue that these categories
provide the correct definition and should be the primary object of study.

\bigskip
Let $K\subset M_\Rr$ and $K^\vee\subset N_\Rr$ be dual reflexive Gorenstein cones.
Let
$$c:K_{(1)}\to \Cc,~~c^\vee:K^\vee_{(1)}\to \Cc
$$
be generic coefficient functions. To this data one should be able to associate (in some vague physical sense)
$N=(2,2)$ superconformal field theories of type IIA and IIB. The switch between the IIA and IIB should
correspond to the switch of the data and the dual data. For the purposes of the following discussion we will focus
on the IIB theory. We should view this theory as some kind of generalized Landau-Ginzburg theory with the
potential $c$ and the generalized K\"ahler data (or mirror potential) given by $c^\vee$.

\bigskip
A reasonably well understood feature of type IIB superconformal field theory is the triangulated category of
the boundary conditions on the open strings. If the target is a Calabi-Yau manifold, then this is simply the
derived category of coherent sheaves on it. In what follows we propose a definition of such triangulated category
in our setting.

\bigskip
Consider the singular affine toric variety ${\rm Spec}(\Cc[K\cap M])$ and the hypersurface in it
$$
X_c:={\rm Spec}(\Cc[K\cap M]/\la \sum_{m\in K_{(1)}} c(m)[m]\ra).
$$
We want to define the triangulated category in question  as some resolution of the factorization
category of $X_c$. The grading is given by $\deg^\vee\in K^\vee$. The actual definition is given
below in terms of the Cox construction.

\bigskip
Consider a regular triangulation $\Sigma$ of $K^\vee$ and the corresponding Cox open subset
$$U_\Sigma \subset \Cc^{K^\vee_{(1)}}$$
given by the points ${\bf z}:K^\vee_{(1)}\to \Cc$ with the zero
locus a subset of some cone $\sigma\in \Sigma$.
 There is a group $G$ which acts on
$U_\Sigma$. It is a subgroup of the torus $(\Cc^*)^{K^\vee_{(1)}}$
which consists of ${\bf \lambda}:K^\vee_{(1)}\to \Cc^*$ with the
property that
$$
\prod_{n\in K^\vee_{(1)}} \lambda(n)^{\la m,n\ra} = 1
$$
for all $m\in M$.
The smooth Deligne-Mumford stack $[U_\Sigma/G]$ is a resolution of singularities of ${\rm Spec}(\Cc[K\cap M])$.

\bigskip
The coefficient function $c$ defines a hypersurface $C=0$ in $U_\Sigma$ where
$$
C({\bf z}) = \sum_{m\in K_{(1)}} c(m) \prod_{n\in K^\vee_{(1)}} z(n)^{\la m,n\ra}
$$
is the $G$-invariant polynomial that corresponds to $\sum_m c(m) [m]\in \Cc[K\cap M]$.

\bigskip
The action of $\Cc^*$ on $[U_\Sigma/G]$ that we alluded to before manifests itself in a natural supgroup
$\hat G\supset G$ defined as
\begin{equation}\label{hatg}
\hat G:=\{ { \bf \lambda}:K^\vee_{(1)}\to \Cc^*\Big\vert
\prod_{n\in K^\vee_{(1)}} \lambda(n)^{\la m,n\ra} = 1,~{\rm for~all~}m\in {\rm Ann}(\deg^\vee)\}.
\end{equation}
We have a natural inclusion
$G\subset \hat G$ and the quotient $\hat G/G$ can be identified with $\Cc^*$ by
choosing $m\in M$ with $\la m,\deg^\vee \ra=1$. Then the map
$\hat G\to \Cc^*$ given by
$$
{\bf \lambda} \mapsto \prod_{n\in K^\vee_{(1)}} \lambda(n)^{\la m,n\ra}
$$
is surjective, has kernel $G$ and is independent from the choice of  $m$ above.

\begin{definition}\label{maincat}
We define the triangulated category
$$
D_B(K,c;\Sigma)
$$
as the derived category of matrix factorizations of $([U_\Sigma/\hat
G];C)$ in the sense of \cite[Definition 2.3.2]{BFK12}, which goes
back to \cite{EP15}. It is obtained from pairs of $\hat
G$-equivariant coherent sheaves $\Ff_1,\Ff_2$ on $U_\Sigma$, with
the maps $f:\Ff_1\to\Ff_2,g:\Ff_2\to \Ff_1\otimes \Ll$ whose
composition is multiplication by $C$. Here $\Ll$ is the
$\hat G$-equivariant line bundle on $U_\Sigma$ such that $C$ is its
invariant section. It is then further localized by
acyclic objects, see \cite{BFK12}.
\end{definition}

We would like to argue that thus defined category, also appropriately called the Landau-Ginzburg model,  should be viewed
as the correct mathematical definition of the triangulated category
of type IIB branes on the theory that corresponds to the above
combinatorial data, even if the said theory itself is not defined
mathematically. Our first observation is the result of Ballard,
Favero and Katzarkov {\cite{BFK12}}, which clarifies the
earlier work of Herbst-Walcher \cite{HW}.

\begin{theorem}\label{samecat}
The category $D_B(K,c;\Sigma)$ does not depend on $\Sigma$ in the sense that
for any two regular simplicial fans $\Sigma_+$ and $\Sigma_-$  as above there is an equivalence of triangulated categories
$$
D_B(K,c;\Sigma_+)=D_B(K,c;\Sigma_-).
$$
\end{theorem}

\begin{proof}
Since $\Sigma_\pm$ are regular fans, there exist functions $\psi_\pm:K^\vee_{(1)}\to\Rr$ such 
that the bottom of the convex hull of $\{(v,\psi_\pm(v)),v\in K^\vee_{(1)}\}$ in $N_\Rr\oplus \Rr$ 
gives $\Sigma_\pm$.
We then consider 
$\psi_t=t\psi_++(1-t)\psi_-$. As $t$ varies from $0$ to $1$, we get a finite list of fans $\Sigma_t$ that
interpolate between $\Sigma_-$ and $\Sigma_+$. By perturbing $\psi_\pm$ slightly, we may assume
that all of the degenerate fans have exactly one nonsimplicial cone, i.e. that we have (a finite list of)
simple toric flops. Thus it suffices to assume that $\Sigma_+$ and $\Sigma_-$ differ by one
simple flop.

\smallskip
This means that there is a subset  $\{v_1,\ldots,v_m,u_1,\ldots,u_n\}\subseteq K^\vee_{(1)}$ such
that the only difference between $\Sigma_+$ and $\Sigma_-$ is due to
a different way they subdivide $\Rr_{\geq 0}\{v_1,\ldots,v_m,u_1,\ldots,u_n\}$.
Specifically, the cones in $\Sigma_+$ involve all $v_i$ and all but one $u_j$; the cones of $\Sigma_-$
involve all but one $v_i$ and all $u_j$.\footnote{The cones may involve other elements of $K^\vee_{(1)}$.}
There is a unique linear relation on $v_i$ and $u_j$
\begin{equation}\label{onepar}
\sum_{i=1}^m \alpha_i v_i  + \sum_{j=1}^n\beta_j u_j = {\bf 0}
\end{equation}
where $\alpha_i\in \Zz_{>0}$ and $\beta_j\in \Zz_{<0}$.

\smallskip
We now consider the one-parameter subgroup of $G$ (and $\hat G$)
given by $\lambda(v_i)=t^{\alpha_i}$ and
$\lambda(u_j)=t^{\beta_j}$ with $t\in \Cc^*$. Other
coordinates $\lambda(v)$ are set to $1$. It is easy to see that it
fits into the setup of Section 3 of \cite{BFK12}. The parameter
$\mu$ of \cite[Theorem 3.5.2]{BFK12} is given by
$$
\sum_{i=1}^m \alpha_i + \sum_{j=1}^n\beta_j.
$$
We see it is equal to zero by taking a pairing of
\eqref{onepar} with $\deg$.

\smallskip
Therefore, this wall-crossing fits the second (flop) case of \cite[Theorem 3.5.2]{BFK12}  and we
get the equivalence of the corresponding factorization categories.
\end{proof}

\begin{remark}
As expected, the triangulated category of the data $(K,K^\vee;c,c^\vee)$ is independent of $c^\vee$. However,
there should be some, yet unknown, construction of the family of such categories  as $c^\vee$ varies. This
family of categories should have a flatness property. Then the above equivalences correspond to the
path in the K\"ahler moduli space of $c^\vee$ between the large K\"ahler limit points that correspond to $\Sigma_+$
and $\Sigma_-$.
\end{remark}

Importantly, we can relate the above defined category $D_B(K,c)$ to the derived category of a Calabi-Yau complete
intersection for any choice of a decomposition
$$
\deg^\vee = t_1+\cdots+t_k.
$$
Specifically, there is the following result,
 due to multiple authors, see \cite{FK14, Isik12,Shi12}.

\begin{theorem}\label{isikthm}
Let $\deg^\vee=t_1+\cdots+t_k$ be a decomposition of $\deg^\vee$ as
in Section \ref{section2}. Let $\Sigma$ be a regular simplicial  fan
in $K^\vee$  considered in that section and $\Xx_{c;\Sigma}$ be the
corresponding complete intersection. Then
$$
D_B(K,c;\Sigma)=D^b(coh-\Xx_{c;\Sigma})
$$
in the sense of equivalence of triangulated categories.
\end{theorem}

\begin{remark}
We can thus view the category $D_B(K,c)=D_B(K,c;\Sigma)$ as a
primary object of interest.  In a somewhat vague sense we view the above Theorem \ref{isikthm}
as large K\"ahler limit description of $D_B(K,c)$ as the size
of the coefficients $c^\vee(t_i)$
for the mirror coefficient function
$$
c^\vee : K^\vee_{(1)}\to\Cc
$$
is large compared to the other values. In what follows we will describe
another explicit geometric realization of $D_B(K,c)$, this time
coming from a more complicated decomposition of $\deg^\vee$ which
corresponds to a more complicated large K\"ahler limit.
\end{remark}

\begin{remark}
If the centrality assumption on the fan \eqref{central} does not hold, then we expect
some interesting structures along the lines of exoflops of Aspinwall \cite{A09,A15}.
However, these are not the focus of the current paper.
\end{remark}

\begin{remark}
There must be a relation between the Hochschild cohomology of the
category $D_B(K,c)$ and the stringy cohomology of
$(K,K^\vee;c,c^\vee)$ in \cite{Borisov.stringy}.
However, we can not presently formulate a precise conjecture,
beyond the basic equality of dimensions.
\end{remark}

\section{Review of Kuznetsov's Clifford double mirrors of complete intersections of quadrics in $\Cc\Pp^n$.}\label{K.recall}
The goal of this section is to review the construction due to
Kuznetsov of noncommutative (Clifford) double mirrors of complete
intersections of quadrics in projective spaces. One
can view this paper as an extension of Kuznetsov's
construction to more general toric varieties, as well as an
explanation of the combinatorics behind it. We are interested in the
Calabi-Yau case of the construction, so we will be working with the
intersection of $k$ quadrics in $\Cc\Pp^{2k-1}$.

\medskip
Let $f_1,\ldots, f_k$ be generic homogeneous degree two polynomials
in the variables {$\xx:=(x_1: \ldots
:x_{2k})$}. Consider the complete intersection $Y\subset
\Cc\Pp^{2k-1}$
$$
Y=\{f_1({\bf x})=f_2({\bf x})=\ldots=f_k({\bf x})=0\}.
$$
The variety $Y$ is smooth and Calabi-Yau by Bertini's theorem and the adjunction formula.
Note that the Cayley hypersurface of this complete intersection $Y$ may be thought of as the
generic bi-degree $(1,2)$ divisor in $\Cc\Pp^{k-1}\times \Cc\Pp^{2k-1}$ given by
$$
u_1f_1({\bf x}) + \cdots + u_k f_k({\bf x}) = 0.
$$
We denote this hypersurface by $X$.

\medskip
The double mirror noncommutative variety can be described as follows.
The polynomial
$$
C({\bf u},{\bf x})=u_1f_1({\bf x}) + \cdots + u_k f_k({\bf x})
$$
allows one to define a graded noncommutative
ring which is the quotient of the free ring
in $k$  commuting central variables $u_i$ and $2k$ noncommuting
variables $y_j$
$$
\Aa =\Cc[u_1,\ldots,u_k]\{y_1,\ldots,y_{2k}\}/\la (\sum_{i=1}^{2k} x_i y_i)^2 + C({\bf u},{\bf x}),~{\bf x}\in \Cc\Pp^{2k-1}\ra.
$$
This is a finitely generated algebra over the homogeneous coordinate ring $\Cc[u_1,\ldots, u_k]$
of $\Cc\Pp^{k-1}$. It is equipped with a half-integer grading such that the degree of $u_i$ is $1$ and the degree of $y_j$
is $\frac 12$. By localizing for the central variables $u_i$ and taking the degree zero component
we get a sheaf of even Clifford algebras $\Bb_0$ over $\Cc\Pp^{k-1}$. Specifically, if we introduce a vector
bundle $\Ee=\Oo^{n+1}$ over $\Cc\Pp^{k-1}$ then $\Bb_0$
is the direct sum of vector bundles
$$
\Bb_0=\Oo \oplus ( \wedge^2\Ee)(1)  \oplus \ldots \oplus( \wedge^{2k}\Ee)(k)
$$
with a certain Clifford multiplication structure.
\footnote{We identify the exterior algebra with a Clifford algebra via the composition
of the embedding into tensor algebra and projection.}

\medskip
The following key result is due to Kuznetsov. It generalizes the work of Kapranov \cite{Kap89} and relates the bounded derived category of the Cayley
hypersurface $X$ in $\Cc\Pp^{k-1}\times \Cc\Pp^n$ with the bounded derived category of sheaves of
$\Bb_0$-modules on $\Cc\Pp^{k-1}$. It is the consequence of Theorem 4.2 in \cite{Kuz08}.
\begin{theorem}
Denote by $p:X\to \Pp^{k-1}$ the natural projection with quadric fibers of dimension
$n-1$. The derived category $D^b(X)$ admits a semiorthogonal
decomposition
$$
D^b(X) = \la D^b(\Pp^{k-1},\Bb_0), p^*D(\Pp^{k-1})\otimes \Oo_{X/\Pp^{k-1}}(1),
\ldots,
p^*D(\Pp^{k-1})\otimes \Oo_{X/\Pp^{k-1}}(2k-2)
\ra
$$
in the sense that the orthogonal complement of $\la  p^*D(\Pp^{k-1})(1),
\ldots,
p^*D(\Pp^{k-1})(2k-2)\ra$ is naturally equivalent to $D^b(\Pp^{k-1},\Bb_0)$.
\end{theorem}

It was shown by Kuznetsov that $D^b(\Pp^{k-1},\Bb_0)$ is equivalent to $D^b(Y)$. Instead of Kuznetsov's
original proof that uses Lefschetz decompositions and homological projective duality, we will sketch
a proof of this statement that will generalize later to more sophisticated examples.

\begin{theorem}
There is an equivalence of categories
$$
D^b(\Pp^{k-1},\Bb_0)=D^b(Y).
$$
\end{theorem}

\begin{proof}
We can relate the derived category of $X$ to the category of
factorizations of the corresponding affine bundle over $\Pp^{k-1}$.
Specifically, let $X_+\subset V_+$  be the singular quadric in
$$
V_+=\Cc\Pp^{k-1}\times \Cc^{2k}
$$
given by $C({\bf u},{\bf x})=0$. This variety $X_+$ admits a $\Cc^*$ action which scales
${\bf x}$. We consider the category of factorizations of $C$ on $U_+$
$$
D^b(coh[V_+/\Cc^*],C).
$$
The relative version of the famous theorem of Orlov \cite{Orl09} gives a semiorthogonal
decomposition of $D^b(X)$
$$
D^b(X) = \la  p^*D(\Pp^{k-1})(1),\ldots,
p^*D(\Pp^{k-1})(2k-2), D^b(coh[V_+/\Cc^*],C)\ra
$$
because the Gorenstein parameter $a$ is given by $(2k-2)$ by adjunction formula.
This implies that
$$
D^b(coh[V_+/\Cc^*],C)=D^b(\Pp^{k-1},\Bb_0).
$$

\medskip
For the second step, which is a particular case of Theorem \ref{samecat}, 
the work of Ballard, Favero and Katzarkov \cite{BFK12} allows one
to pass from $X_+\subset V_+$ to the hypersurface $X_-\subset V_-=\Cc^k\times \Cc\Pp^{2k-1}$ given by
$$
C({\bf u},{\bf x})=0,
$$
together with the action of $\Cc^*$ on ${\bf u}$. Indeed, we may
consider two open subsets $U_\pm\subset \Cc^k\times \Cc^{2k}$
defined by ${\bf x}\neq {\bf 0}$ and ${\bf u}\neq {\bf 0}$
respectively. We define the group $\hat G=\Cc^*\times\Cc^*$ that
scales both sets of coordinates. Then the categories of
factorizations
$D^b(coh[V_+/\Cc^*],C)$ and
$D^b(coh[V_-/\Cc^*],C)$ are equivalent to those of $D^b(coh[U_+/\hat
G],C)$ and $D^b(coh[U_-/\hat
G],C)$ respectively. We then have by
\cite{BFK12}
$$
D^b(coh[U_+/\hat G],C)=D^b(coh[U_-/\hat G],C)
$$
due to a certain ``conservation of canonical class". Specifically,
the subgroup $G$ of $\hat G$ that preserves $C({\bf u},{\bf x})$ acts by
$$
\lambda_t {\bf u} = t^{-2}{\bf u},~\lambda_t {\bf x} = t{\bf x}.
$$
Therefore, the weight of it on the anticanonical bundle of
$\Cc^{k}\times\Cc^{2k}$ restricted to the fixed point $({\bf 0,\bf
0})$ is $\sum_{i=1}^k(-2)+\sum_{j=1}^{2k}1=0$. This is the condition
needed for the equivalence, see \cite[Theorem 3.5.2]{BFK12}.

\medskip
Thus we get
$$
D^b(\Pp^{k-1},\Bb_0) =D^b(coh[V_-/\Cc^*],C).
$$
Finally, we observe that $D^b(coh[V_-/\Cc^*],C)$ is equivalent to the derived category of the
complete intersection $D^b(Y)$ by the work of Isik \cite{Isik12} and Shipman \cite{Shi12}.
\end{proof}

\section{Clifford double mirror construction.}\label{defcliff}
In this section we generalize the Kuznetsov's example by uncovering
the underlying toric geometry. Specifically, we construct
noncommutative double mirrors of Calabi-Yau complete intersections
in Gorenstein toric varieties, given certain natural combinatorial
data. These noncommutative mirrors consist of  a pair $(\Ss, \Bb_0)$
where $\Ss$ is a smooth toric DM stack and $\Bb_0$ is a sheaf of
algebras which serves as the structure sheaf of
$(\Ss, \Bb_0)$. The coherent sheaves on $(\Ss, \Bb_0)$ are coherent
sheaves on $\Ss$ which are also $\Bb_0$-modules.

\medskip
We work in the notations of Sections \ref{section2} and  \ref{philosophy}. Namely,
we have dual Gorenstein
cones $K$ and $K^\vee$ in lattices $M$ and $N$, with degree elements $\deg$
and $\deg^\vee$. We denote by $k$ the index $\la \deg,\deg^\vee\ra$.
There is also given a generic coefficient function
$$
c:K_{(1)}\to \Cc.
$$

\medskip
\subsection{Definition of $(\Ss,\Bb_0)$.}
The key idea of this paper is that Kuznetsov's and related examples correspond to
the decompositions  of $\deg^\vee$
\begin{equation}\label{dec2}
\deg^\vee = \frac{1}{2}(s_1 + \cdots +
s_{2k})
\end{equation} with $s_i \in K^\vee_{(1)}$. We assume the elements
$s_1,\ldots,s_{2k}$
are $\Rr$-linearly independent. Moreover,
we assume that there exists (and is chosen) a  regular simplicial fan
$\Sigma$
with support $K^\vee$ and rays based on $K^\vee_{(1)}$ such that

\begin{equation}\label{central2}
\emph{All maximum dimensional cones of $\Sigma$ contain $s_1,\ldots,s_{2k}$.}
\end{equation}

We call this \emph{the centrality assumption} on $\Sigma$.
As usual, we will also denote by $\Sigma$ the corresponding simplicial complex
on the set $K^\vee_{(1)}$.
Note that in the case of decompositions that correspond to complete intersections
considered in Section \ref{section2} the analogous centrality condition \eqref{central} can always
be assured. It may no longer be the case in this setting, even for $k=1$.

\medskip
To orient the reader, the idea of our construction is the following.
The centrality assumption \eqref{central2} allows us to view the resolution $\Pp_\Sigma$
of ${\rm Spec}(\Cc[K\cap M])$ as a vector bundle of rank $2k$ over a toric base.
Then the coefficient function $c$ gives a quadric section of this vector bundle
in the sense of Kuznetsov \cite{Kuz08}. This defines a sheaf of even Clifford
algebras over the base of the fibration.

\medskip
Unfortunately, there are some inevitable technical difficulties
that need to be overcome
to make the above picture precise. First of all, we need to work with smooth
toric DM stacks, rather than schemes. Second, we need to be careful in our choice
of the lattice for the base of the fibration as described below.

\medskip
Let $U_\Sigma$ be the Cox open subset of $\Cc^{K^\vee_{(1)}}$ which
consists of maps ${\bf z}:{K^\vee_{(1)}}\to \Cc$ such that the
preimage of $0$ is an element of the simplicial complex $\Sigma$.

\medskip
We define an abelian group
$$
\overline N : = N / \Zz s_1 + \cdots + \Zz s_{2k}+ \Zz \deg^\vee .
$$
Notice that we also quotient by $\deg^\vee$ in the last component. Otherwise, $\overline N$
would always have an order two torsion element which is the image of $\deg^\vee$.
However, the abelian group $\overline N$ may still have torsion elements as explained in Remark \ref{remark: non saturatedness}. Next, we consider the stacky fan $\overline\Sigma$ in $\overline N$ which corresponds to the simplicial
complex in $K^\vee_{(1)}-\{s_1,\ldots,s_{2k}\}$ whose maximum sets are
obtained from those of $\Sigma$ by removing all of $s_i$.
We immediately observe that the centrality condition \eqref{central2} implies that the natural identification
$$
\Cc^{K^\vee_{(1)}} = \Cc^{K^\vee_{(1)}-\{s_1,\ldots,s_{2k}\}}
\times\Cc^{2k}
$$
induces the natural identification
$$
U_\Sigma = U_{\overline \Sigma}\times \Cc^{2k}.
$$

\medskip
We will now discuss the various groups associated to the construction. Recall that
in Section \ref{philosophy} we considered two subgroups $G$ and $\hat G$
of $(\Cc^*)^{K^\vee_{(1)}}$ defined by
$$
G:=\{ { \bf \lambda}:K^\vee_{(1)}\to \Cc^*\Big\vert
\prod_{n\in K^\vee_{(1)}} \lambda(n)^{\la m,n\ra} = 1,~{\rm for~all~}m\in M\}
$$
$$
\hat G:=\{ { \bf \lambda}:K^\vee_{(1)}\to \Cc^*\Big\vert
\prod_{n\in K^\vee_{(1)}} \lambda(n)^{\la m,n\ra} = 1,~{\rm for~all~}m\in {\rm Ann}(\deg^\vee)\}.
$$
The group $G$ is the group that corresponds to the smooth toric DM  stack $\Pp_{\Sigma}
=[U_\Sigma/G_\Sigma]$,
and the group $\hat G$ defines a $\Cc^*$ action on $\Pp_{\Sigma}$.

\medskip
There is an additional group of interest $H$ isomorphic to $\Cc^*$ given by
$$
\lambda(s_i) = t, \lambda(v)=1,~{\rm ~for~all~}v \in
K^\vee_{(1)}-\{s_1,\ldots,s_{2k}\}
$$
with $t\in \Cc^*$. Notice that $H\subseteq \hat G$, because for any $m\in {\rm Ann}(\deg^\vee)$
there holds
$$
\prod_{n\in  K^\vee_{(1)} }\lambda(n)^{\la
m,n\ra} = t^{\sum_i \la m, s_i
\ra} =t^{\la m,2\deg^\vee \ra} = 1.
$$
Analogous calculation shows that $H\cap G=\{\pm 1\}$.

\begin{remark}\label{splitHG}
The group $H$ acts by scaling the fibers of the $\Cc^{2k}$ fibration
$U_{\Sigma}\to U_{\overline\Sigma}$. Note also that the inclusion $H\subseteq \hat G$
is split, since one can consider the evaluation at $s_1$ as the
splitting map $\hat G\to H$. While such splitting is not completely natural, as it requires a choice of one of $s_i$, it will suffice for our purposes.
\end{remark}

\medskip
Notice that there is a natural map $\hat G/H\to (\Cc^*)^{K^\vee_{(1)}-\{s_1,\ldots,s_{2k}\}}$ since the coordinates of $H$ that correspond
to $K^\vee_{(1)}-\{s_1,\ldots,s_{2k}\}$ are equal to $1$.
This gives an action of $\hat G/H$ on $U_{\overline\Sigma}$. We observe that this is
precisely the action used in the definition of the toric DM stack $\Pp_{\overline\Sigma}$.
\begin{lemma}
We denote by $\overline G$ the quotient group $\hat G/H$. Then
the toric DM stack associated to $\overline\Sigma$ in $\overline N$
is given by $[U_{\overline\Sigma}/{\overline G}]$.
\end{lemma}

\begin{proof}
According to \cite{BCS05} we need to identify $\hat G/H$ with the character group of the derived Gale dual of the complex
\begin{equation}\label{togale}
0\to\Zz^{K^\vee_{(1)}-\{s_1,\ldots,s_{2k}\}}\to \overline N\to 0
\end{equation}
defined by looking at linear combinations of images in $\overline N$ of degree one elements of $K^\vee$.

\medskip
To define the derived Gale dual one needs to resolve $\overline N$ by free groups.
Even in the case when it is already free, it will be convenient to consider
the short exact sequence of free abelian groups
$$
0\to L \to N/\Zz \deg^\vee \to \overline N\to 0.
$$
Note that $\Zz\deg^\vee$ is saturated, since $\deg^\vee$
is the smallest  lattice element in the interior of $K^\vee$.
Thus $N/\Zz \deg^\vee$ is free.
The sublattice $L$ of $N/\Zz\deg^\vee$ is generated by the images of
$s_1,\ldots,s_{2k}$ and is naturally isomorphic to the quotient of $\Zz^{2k}$
by $\Zz(1,\ldots,1)$. We note that $L$ might not be a saturated sublattice of $N/\Zz \deg^\vee$.

\medskip
Following the definitions of \cite{BCS05} we now replace the complex \eqref{togale} above by a quasiisomorphic
complex of free groups
$$
0\to \Zz^{K^\vee_{(1)}-\{s_1,\ldots,s_{2k}\}}\oplus L
 \to N/\Zz\deg^\vee \to 0.
$$
The group that corresponds
to the stacky fan $\overline \Sigma$ in $\overline N$ is given as the character group of the cokernel $L_1$ of the (injective)
dual map
$$
{\rm Ann}(\deg^\vee)\to \Zz^{K^\vee_{(1)}-\{s_1,\ldots,s_{2k}\}}\oplus L^\vee.
$$
We have the exact sequence (in the multiplicative notation)
$$
1\to {\rm Hom}(L_1,\Cc^*) \to (\Cc^*)^{K^\vee_{(1)}-\{s_1,\ldots,s_{2k}\}}
\times (\Cc^*)^{2k}/\{t,\ldots,t\}
\to {\rm Hom}({\rm Ann}(\deg^\vee),\Cc^*).
$$
When compared with the definition of $\hat G$ via
$$
1\to \hat G \to (\Cc^*)^{K^\vee_{(1)}-\{s_1,\ldots,s_{2k}\}}
\times (\Cc^*)^{2k}
\to {\rm Hom}({\rm Ann}(\deg^\vee),\Cc^*)
$$
we recover the needed natural isomorphism between ${\rm Hom}(L_1,\Cc^*)$
and ${\overline G}=\hat G/H$. It is also easy to see that the action on $U_{\overline\Sigma}$
is induced by the map ${\overline G}\to  (\Cc^*)^{K^\vee_{(1)}-\{s_1,\ldots,s_{2k}\}}$.
\end{proof}

We now return to geometry by recalling the definition of a quadric fibration in the sense of
 \cite[Section 3]{Kuz08}.
It is given by the
following data:
\begin{itemize}
\item a smooth algebraic variety $S$;
\item a vector bundle $E\to S$;
\item  a line bundle $\Ll$ on $S$;
\item  an embedding of vector bundles $\sigma:\Ll\to Sym^2(E^\vee)$.
\end{itemize}
This data defines $\pi: \Pp_S(E)\to E$ the projectivization of $E\to
S$. The data of $\sigma$ gives a section of $H^0( \Pp_S(E),
\Oo(2)\otimes \pi^* \Ll^\vee)$ where $\Oo(1)$ be the dual of
the  tautological
line bundle on $\Pp_S(E)$. Kuznetsov denotes by $\Xx\subset
\Pp_S(E)$ the zero locus of $\sigma$ in $ \Pp_S(E)$. Then  the
restriction of $\pi$ to $\Xx$ denoted by $p:\Xx\to S$ is a flat
fibration with (possibly singular) quadric fibers. The embedding
assumption above is crucial. It is equivalent to the flatness of the
fibration $p:\Xx\to S$.

\medskip
In our setting,
the coefficient function $c$ gives a global function on $U_\Sigma$
given by
$$
C({\bf z}) = \sum_{m\in K_{(1)}} c(m) \prod_{n\in K^\vee_{(1)}} z(n)^{\la
m,n\ra}.
$$
We observe that $C$ naturally fits into the definition of quadric fibration as above
which is $\overline G$-equivariant.
\begin{proposition}
The zero set ${\{C=0\}}$ defines a $\overline
G$-equivariant quadric fibration over $U_{\overline\Sigma}$.
\end{proposition}

\begin{proof}
It follows immediately from the definition of $C$ that it is
invariant under $G$ and is semi-invariant
under $\hat G$. Specifically, the action of $\lambda$ on $C$ scales
each summand by
$$
\prod_{n\in K^\vee_{(1)}}\lambda(n)^{\la m,n\ra}.
$$
Since different $m\in K_{(1)}$ differ by an element of ${\rm
Ann}(\deg^\vee)$, the above term is independent of $m$. Moreover,
for ${\bf \lambda}\in H$ the above term is $t^{\la m, \sum_i
s_i\ra}=t^{\la m,2\deg^\vee\ra}=t^2$. In other words, $C({\bf z})$ has
total degree $2$ in the variables $z(s_1),\ldots, z(s_{2k})$.
\end{proof}

\begin{remark}\label{flatornot}
It is important to point out that  flatness of the quadric fibration
can not be taken for granted. We will see later in Section
\ref{noflatness} that it is not always the case. Since the fibration
is defined by a hypersurface $C=0$, the geometric criterion for
flatness is that all of the fibers are hypersurfaces, i.e. for all points ${\bf \overline z} \in
U_{\overline\Sigma}$ the restriction of $C$ to the fiber of $U_\Sigma\to U_{\overline\Sigma}$
is not
identically zero. We will refer to this condition on our
combinatorial data as the flatness assumption. If true, it can be
typically established by Bertini theorem arguments, so we will
tacitly assume it, unless stated otherwise.
\end{remark}

We are now ready to define the noncommutative variety $(\Ss,\Bb_0)$
that corresponds to the decomposition \ref{dec2}, fan $\Sigma$ and the choice of
the coefficient function $c$.
\begin{definition}\label{defineb0}
The noncommutative variety $(\Ss,\Bb_0)$ is the sheaf of even Clifford algebras $\Bb_0$
over the smooth DM stack $\Ss=[U_{\overline\Sigma}/\overline G]$ associated
to the quadratic function $C({\bf z})$ of the stacky bundle
$$
[U_{\Sigma}/\overline G]\to [U_{\overline\Sigma}/\overline G]
$$
where we use the splitting of $\hat G\to \overline G$ to define the
action of $\overline G$ on $U_{\Sigma}$.
\end{definition}


\begin{remark}\label{explicitdefcliff}
More explicitly, the category of coherent sheaves on $(\Ss,\Bb_0)$ is defined
as the category of $\overline G$-equivariant sheaves over the even part of the
(locally constant) sheaf of Clifford algebras over
$U_{\overline\Sigma}$ given by
\begin{equation}\label{defbs}
\Big(\Oo_{U_{\overline\Sigma}}\{y_1,\ldots,y_{2k}\}/\la (\sum_{i=1}^{2k} z_iy_i)^2+C({\bf z}),~~{\rm for~all~}z_1,\ldots,z_{2k} \ra\Big)_{even}
\end{equation}
where $y_1,\ldots, y_{2k}$ are free noncommuting variables and $z_i$
is a shorthand notation for $z(s_i)$. The subscript \emph{even}
refers to the parity of the number of $y_i$. Here the action of $\overline G$ on
$y_1,\ldots,y_{2k}$ is defined as follows. For an element $\overline \lambda \in \overline G$
consider the lift $\lambda$ to $\hat G$ from the splitting $\hat G=\overline G\times H$. Denote
by $\varphi(\lambda)$ the image of $\lambda$ in $\hat G/G =\Cc^*$. Then we define
$$
\overline \lambda(y_iy_j) = \lambda_i^{-1}\lambda_j^{-1}\varphi(\lambda) y_iy_j
$$
where $\lambda_i$ is the coordinate of $\lambda$ that corresponds to $s_i$.
This definition is ensures that
$(\sum_i z_i y_i )^2+C({\bf z})$ is semi-invariant with respect to $\overline G$
with character $\varphi$.
\end{remark}

\begin{remark}
One could in principle define the action of $\overline G$ on the sheaf of Clifford algebras
without a choice of the splitting $\hat G=\overline G \times H$. For any lift
$\lambda\in \hat G$ of $\overline \lambda\in \overline G$ the formula
$$
\overline \lambda(y_iy_j) = \lambda_i^{-1}\lambda_j^{-1}\varphi(\lambda) y_iy_j
$$
gives the same result. Indeed, for an element $\lambda \in H=\Cc^*$ that corresponds to the complex number $t$,
we have $\varphi(\lambda)=t^2$ and $\lambda_i=\lambda_j=t$, so
the right hand side is $1$. However, we find it convenient to pick a
splitting.
\end{remark}

\medskip
We will now describe some simple examples of our construction. In particular, we make a connection
to the  Kuznetsov's example that we studied in Section \ref{K.recall} to show how it fits
into the above toric formalism.
For the sake of simplicity, we will start with
the classical example of $(2,2,2,2)$-complete intersections in
$\Cc\Pp^7$ which has been  thoroughly
studied by multiple authors (see  \cite{Add09, CDHPS10}). It corresponds to the
$k=4$ case of Section \ref{K.recall}.

\medskip
Afterwards, we will consider a free $\Zz_2$
quotient of the above example which is a family of Calabi-Yau threefolds
with Hodge numbers $(1,33)$.
Its Clifford double mirror could be obtained completely analogously but appears to
be new.

\subsection{Example: $(2, 2, 2, 2)$-complete intersections in $\Cc\Pp^7$.}\label{2222}

We start with the lattice $N_1\cong \Zz^7$ of the fan of $\Cc\Pp^7$. We try to keep the construction
as natural as possible, in particular, we try to keep it symmetric with respect to the permutations
of coordinates of $\Cc\Pp^7$. Thus we view  this lattice as the quotient of the lattice
$\bigoplus_{i=1}^8 \Zz e_i$ by the sublattice
$\Zz (\sum_{i=1}^8 e_i)$. The fan of $\Cc\Pp^7$ has maximum cones that are generated by subsets of $7$ out of $8$
elements of $\{e_i\}$. The rays of the fan are generated  by $e_i$ and correspond to coordinate hyperplanes $D_i\subset \Cc\Pp^7$.

\medskip
The dual lattice $M_1$ is naturally identified with the rank $7$ sublattice of $\bigoplus \Zz e_i^\vee$ with the property that the sum
of the entries is $0$. Here we use $\{e_i^\vee\}$ to denote the dual basis of $\{e_i\}$.

\medskip
To consider the complete intersection of four general quadrics in $\Cc\Pp^7$, we, as usual, subdivide the standard anitcanonical divisor
of $\Cc\Pp^7$ as
\[
(D_1+D_2) + (D_3+D_4) + (D_5+D_6) + (D_7+D_8) = -K_{\Cc\Pp^7}.
\]
We introduce the extended lattices $M=M_1\oplus \Zz^4$ and
$N=N_1\oplus \Zz^4$  and consider
the reflexive Gorenstein cone $K^\vee$ in $N$ generated by the
lattice elements
\[\begin{split}
&s_1=(e_1;1,0,0,0),s_2=(e_2;1,0,0,0),t_1=({\bf 0};1,0,0,0),\\
&s_3=(e_3;0,1,0,0),s_4=(e_4;0,1,0,0),t_2=({\bf 0};0,1,0,0),\\
&s_5=(e_5;0,0,1,0),s_6=(e_6;0,0,1,0),t_3=({\bf 0};0,0,1,0),\\
&s_7=(e_7;0,0,0,1),s_8=(e_8;0,0,0,1),t_4=({\bf 0};0,0,0,1).
\end{split}
\]
The above $s_i$ and $t_j$ form the set of degree one elements $K^\vee_{(1)}$.

\medskip
The dual cone $K$ in $M$ is generated by $32$ elements
\[\begin{split}
&(-e_1^\vee-e_2^\vee+2e_i^\vee;1,0,0,0),(-e_3^\vee-e_4^\vee+2e_i^\vee;0,1,0,0),\\
&(-e_5^\vee-e_6^\vee+2e_i^\vee;0,0,1,0),(-e_7^\vee-e_8^\vee+2e_i^\vee;0,0,0,1)
\end{split}
\]
for all $i=1,\ldots,8$. The degree $1$ lattice elements that form $K_{(1)}$ are
given by
\[\begin{split}
&(-e_1^\vee-e_2^\vee+e_i^\vee+e_j^\vee;1,0,0,0),(-e_3^\vee-e_4^\vee+e_i^\vee+e_j^\vee
;0,1,0,0),\\
&(-e_5^\vee-e_6^\vee+e_i^\vee+e_j^\vee;0,0,1,0),(-e_7^\vee-e_8^\vee+e_i^\vee+e_j^\vee;
0,0,0,1)
\end{split}
\]
for all $36$ unordered pairs $(i,j)$. The corresponding coefficient function
$$
c:K_{(1)}\to \Cc
$$
encodes the coefficients of $4$ quadrics at the standard monomials $x_ix_j$.

\medskip
In what follows, we will adapt a somewhat different way of looking at $K$ and $K^\vee$.
We can think of the lattice $N$ as the rank $11$ quotient of $\Zz^{12}$ by the sublattice $\Zz(\sum_{i=1}^8s_i-2\sum_{j=1}^4 t_j)$.
The cone $K^\vee$ is then just the image of  the positive orthant $\Zz_{\geq 0}^{12}$. The dual lattice $M$ can be viewed as a corank one
sublattice of $\bigoplus_{i=1}^8\Zz s_i^\vee \oplus \bigoplus_{j=1}^4 \Zz t_j^\vee$, namely
\[
M=\{\sum_{i=1}^8 a_i s_i^\vee + \sum_{j=1}^4 b_j t_j^\vee \mid \sum_ia_i=2\sum_jb_j\}.
\]
The cone $K$ is then the intersection of $M$ with the positive orthant.

\medskip
The degree element $\deg^\vee$ in $K^\vee$ is given by
\begin{equation}\label{equation: two ways}
\deg^\vee =\frac 12 \sum_{i=1}^8 s_i = \sum_{j=1}^4 t_j.
\end{equation}
The degree element $\deg\in K$ is $\sum_{i=1}^8 s_i^\vee+  \sum_{j=1}^{4}t_j^\vee$.
The elements of $K^\vee_{(1)}$ are the aforementioned $s_i$ and $t_j$. The elements of $K_{(1)}$
are of the form
$s_i^\vee+s_j^\vee+t_k^\vee$ where $i$ may or may not equal $j$. When $i=j$ we get the above
$32$ generators of the rays of $K$.

\medskip
The equation \eqref{equation: two ways} is a prototypical situation
where for two different large K\"ahler limits of the $N=(2,2)$
theories one gets the description of the triangulated category in
two ways. On the one hand, the decomposition
$$
\deg^\vee = \sum_{j=1}^4 t_j
$$
allows one to see this category as the derived category of the  complete intersection of four quadrics in $\Cc\Pp^7$. On the other hand, the decomposition
$$
\deg^\vee =\frac 12 \sum_{i=1}^8 s_i
$$
leads to its description as the derived category of coherent sheaves for the Clifford algebra
over $\Cc\Pp^3$.

\medskip
For the first construction, we consider the regular simplicial fan on $K^\vee$ with eight maximum
cones obtained by removing one of the eight elements $s_i$. It is easy to see that the resulting
complete intersection is that of four quadrics in $\Cc\Pp^7$. The coefficient of the
$x_ix_j$ monomial of the  $k$-th quadric
is the coefficient $c(s_i^\vee+s_j^\vee+t_k^\vee)$.

\medskip
For the second construction, we use the fan $\Sigma$ in $K^\vee$ whose maximal
cones are given by eight elements $s_i$ and three out of four $t_i$.
We observe that the group $\hat G$ described in \eqref{hatg}
is given by $\Cc^*\times \Cc^*$ with the action
$$
(\lambda_1,\lambda_2)(t_j) = \lambda_1;~
(\lambda_1,\lambda_2)(s_i) = \lambda_2.
$$
The map to $\Cc^*$ whose kernel is $G$ is given by
$$
(\lambda_1,\lambda_2)\mapsto \lambda_1\lambda_2^2.
$$
The group $H$ is given by $\{(1,\lambda_2),\lambda_2\in \Cc^*\}\subset \hat G$,
and the group $\overline G=\hat G/H$ is naturally isomorphic to $\Cc^*$ with
the diagonal map to $(\Cc^*)^4$ in view of $(\lambda_1,-)(t_j)=\lambda_1$.
The lattice $\overline N$ is the quotient of $N$ by the lattice generated by
$s_i$ and $\deg^\vee$. Therefore, it can be naturally identified with the quotient of
$\oplus_{j=1}^4\Zz t_j$ by the span of $\sum_j t_j$. The fan $\overline\Sigma$
is the standard fan of $\Cc\Pp^3$.

\medskip
Then it is easy to see that the sheaf of Clifford algebras constructed in
Definition \ref{defineb0} is the same sheaf on $\Cc\Pp^3$ as the one constructed in Section \ref{K.recall}.

\begin{remark}\label{octic}
The Clifford noncommutative variety $(\Bb,\Ss_0)$ should be viewed as a crepant resolution
of the singular double cover of $\Cc\Pp^3$ ramified over the determinantal octic which is
the determinant of the symmetric $8\times 8$ matrix of degree $1$ forms encoded by $c$.
See \cite{KuzICM} for details.
\end{remark}

\subsection{Free involution quotients of  $(2,2,2,2)$-complete intersections.}\label{2222tau}

In this example, we consider complete intersections of four quadrics in $\Cc\Pp^7$ which admit a free $\Zz_2$ action.
This is a particular case of more general construction of such complete intersections with more sophisticated free group
actions, see \cite{Bea98, Hua11}. While the more interesting non-abelian actions can not be easily realized in the toric setting
of this paper,
\footnote{One can try to extend to such actions by looking at automorphisms of the fan,
similar to  \cite{Stapledon}.} 
the simple case of an involution fits nicely into our construction.

\medskip
An involution $\tau$ on a smooth complete intersection  $X$ of four quadrics in $\Cc\Pp^7$ always comes from a linear action on $\Cc^8$. In view of holomorphic Lefschetz formula (or by direct inspection of fixed point sets), if this involution is fixed point free on $X$, then it must act with trace $0$ on $H^0(X,\Oo(1))$. Without loss of generality we may assume that
this involution $\tau$ acts by
\[
\tau (x_1:x_2:\ldots:x_8) = (-x_1:-x_2:-x_3:-x_4:x_5:x_6:x_7:x_8).
\]
Then the action of $\tau$ on the space of degree two polynomials in $x_i$ has $1$ and $(-1)$-eigenspaces of dimension $20$ and $16$ respectively.
Holomorphic Lefschetz formula then forces us to consider $4$ $\tau$-invariant quadrics in $\Cc\Pp^7$, so that the action of $\tau$ on
$H^0(X,\Oo(2))$ has eigenspaces of dimension $16$ each, see \cite{Hua11}.

\medskip
In terms of toric geometry, the situation is extremely similar to the setting of the previous example, except that the lattice $\Zz^8/\Zz\sum_{i=1}^8e_i$
is now replaced by a suplattice of index two to include $\frac 12 (e_1+e_2+e_3+e_4)$. On the dual side, we must consider the corresponding sublattice of
index two. When extended to $N$, we now have
\[
N=\Big((\bigoplus_{i=1}^8 \Zz s_i + \frac 12 \Zz (s_1+s_2+s_3+s_4))\oplus \bigoplus_{j=1}^4\Zz t_j\Big)/ \Zz(\sum_{i=1}^8s_i-2\sum_{j=1}^4t_j)
\]
with the cone $K^\vee$ still given as the image of the nonnegative orthant.
The dual lattice $M$ is now given by
\[
M=\{\sum_{i=1}^8 a_i s_i^\vee + \sum_{j=1}^4 b_j t_j^\vee \mid \sum_{i=1}^8a_i=2\sum_{j=1}^4b_j~{\rm and~} \sum_{i=1}^4a_i {\rm ~is~even}\}.
\]
The cone $K$ is the intersection of $M$ with the positive orthant.

\medskip
The degree elements $\deg^\vee$ and $\deg$ are given by the same formulas as before. The set $K^\vee_{(1)}$ is unchanged, however the set
$K_{(1)}$ is now smaller. It consists of elements of the form
\[
s^\vee_i+s^\vee_j+t^\vee_k
\]
where $i$ and $j$ are either both in the range of $1,\ldots,4$ or are both in the range $5,\ldots,8$. Observe that this is consistent with the choice of quadrics  from the invariant eigenspace of $\tau$.

\medskip
Again, we now look at the groups $\hat G$, $ G$ and $H$.
The group $\hat G$ is now equal to $\Cc^*\times \Cc^*\times \{\pm 1\}$
which is written in coordinates as
$$
(\lambda_1,\lambda_2,\pm 1) (t_j) = \lambda_1,~
(\lambda_1,\lambda_2,\pm 1) (s_i) =\lambda_2, ~{\rm if~}1\leq i\leq 4,
$$
$$
(\lambda_1,\lambda_2,\pm 1) (s_i) =\pm\lambda_2, ~{\rm if~}5\leq i\leq 8.
$$
The map $\hat G\to \Cc^*$ still sends
$$
(\lambda_1,\lambda_2,\pm 1) \mapsto \lambda_1\lambda_2^2.
$$
The subgroup $H$ of $\hat G$ that scales the variables that correspond
to $s_i$  is $(1,\Cc^*,1)$. The quotient group $\overline G=\hat G/H$
is naturally identified with $\Cc^*\times \{\pm 1\}$.

\begin{remark}\label{remark: non saturatedness}
In this example, the lattice $\Zz s_1+\Zz s_2+\ldots +\Zz s_8$ is not saturated in $M$, even after adding $\deg^\vee  = \frac 12 \sum_{i=1}^8 s_i$.
Indeed, there is also an element $ \frac 12  (s_1+s_2+s_3+s_4)$ in the real span of $s_i$, which is not an integer linear combination of $s_i$ and $\deg^\vee$.
\end{remark}

\medskip
Let us now discuss the noncommutative variety $(\Ss,\Bb_0)$ in view of the above remark.
The images of $t_1,\ldots,t_4$ still add up to $0$, so the base $\Ss$
is a $\Zz_2$ gerbe over $\Cc\Pp^3$ given by the quotient $[\Cc\Pp^3/\Zz_2]$ by the trivial
$\Zz_2$ action. Alternatively, it is a quotient of $\Cc^4-\{{\bf 0}\}$ by $\Cc^*\times \{\pm 1\}$
where the first factor acts in the usual way and the second factor acts trivially.
The vector bundle used to construct the Clifford algebra is now a bit different. Namely,
it is a direct sum of two rank four bundles on which the extra involution acts as $(-1)$ and $1$
respectively.

\medskip
It is interesting to investigate this case further along the lines of  Remark \ref{octic}. On the
coarse moduli space $\Cc\Pp^3$ of $\Ss$, the symmetric determinantal octic is now a union of
two symmetric determinantal quartics $Q_+$ and $Q_-$, since the corresponding matrix consists of two $4\times 4$ blocks.

\medskip
The fact that we consider a gerbe $[\Cc\Pp^3/\Zz_2]$ can be encoded by looking at the
semidirect product of the even Clifford algebra of Kuznetsov's original construction with
the group algebra $\Cc[h]/\la h^2-1\ra$ of $\Zz_2$. Let us investigate the center of this
algebra in more detail. Let us first localize over the generic point of $\Cc\Pp^3$, i.e.
we will work over the field $F$ of rational functions on $\Cc\Pp^3$. If we diagonalize the
quadratic forms on $V_+$ and $V_-$, then the semidirect product of the even Clifford algebra
with the group ring of $\Zz_2$ gives the even part (in $y$) of the quotient of
the free algebra
$$
F\{y_1^+,\ldots,y_4^+, y_1^-,\ldots,y_4^-,h\}
$$
by the two-sided ideal generated by the relations that $h$ commutes with $y_i^+$ and
anti-commutes with $y_i^-$, the relation $h^2-1$, and the usual Clifford relations
$$
(y_i^+)^2+c_i^+,(y_i^-)^2+c_i^-,
y_i^+y_j^- + y_j^-y_i^+
$$
for all $i$ and $j$ and
$$
y_i^+y_j^+ + y_j^+y_i^+,y_i^-y_j^- + y_j^-y_i^-
$$
for $i\neq j$.
Note that $c_i$ may not assumed to be $1$, since $F$ is not algebraically closed.

\medskip
There is  a grading by $\Zz_2^9$ that
looks at parity of monomials in $y_i^+$, $y_i^-$ and $h$.
In fact, the algebra is easily seen to be
of dimension $2^8$ over $F$ with the basis given by
\begin{equation}\label{twoeight}
h^l\prod_{i\in I}y_i^+ \prod_{j\in J}y_j^-
\end{equation}
for sets $I$, $J$ with $\vert I\vert  + \vert J\vert$ even and $l\in \{0,1\}$.

\medskip
The same grading descends to the center, so to calculate the center we
just need to determine which
monomials in \eqref{twoeight} are central. If one has $I$ which is neither empty nor the whole set
$\{1,\ldots,4\}$, then by taking a commutator with $y_i^+y_j^+$ with $i$ in $I$ and $j$ not in $I$
we get a factor of $(-1)$. The same  happens for $J$. So there are
$8$ monomials to consider as possible elements of the center.
$$
1,
h,
y_1^+\cdots y_4^+,
hy_1^+\cdots y_4^+ ,
y_1^-\cdots y_4^-,
hy_1^-\cdots y_4^- ,
$$
$$y_1^+\cdots y_4^+
y_1^-\cdots y_4^- ,
hy_1^+\cdots y_4^+
y_1^-\cdots y_4^- .
$$
Of these, commutator with $y_1^+y_1^-$ excludes all but
$$
1,h y_1^+\cdots y_4^+ ,hy_1^-\cdots y_4^- ,y_1^+\cdots y_4^+
y_1^-\cdots y_4^-
$$
which are indeed central.
It remains to observe that the relations of the algebra imply that
$$
(h y_1^+\cdots y_4^+ )^2 = c_{1,+}\cdots c_{4,+} = \det(C_+)
$$
and
$$
( hy_1^-\cdots y_4^- )^2 = c_{1,-}\cdots c_{4,+} = \det(C_-)
$$
so at least generically the center of the algebra is the $(\Zz_2)^2$ Galois cover of
$\Cc\Pp^3$ obtained by attaching the square roots of the two quartics $Q_+$ and
$Q_-$.

\medskip
This motivates the following more precise conjecture.
\begin{conjecture}
The Clifford double mirror of Definition \ref{defineb0}
is a crepant categorical resolution of the $\Zz_2\times \Zz_2$ Galois cover
of $\Cc\Pp^3$ ramified at the two quartics $\Qq_+$ and $\Qq_-$.
\end{conjecture}

In support of this conjecture, let us calculate the stringy Euler numbers of this
Galois cover. Each quartic $Q_\pm$  has $10$ nodes and Euler characteristics
$14$. They intersect in a smooth curve $Y$ with $2g-2=64$, so $\chi(Y)=-64$.
We have the following Euler characteristics
$$
\chi(Y)=-64, \chi(Q_\pm-Y) = 78, \chi(\Cc\Pp^3-(Q_+\cup Q_-))=-88.
$$
On the Galois cover, the first kind of points gives preimage of $1$ point,
the second  gives $2$ points, and the last gives $4$ points.
Therefore, we get the usual Euler characteristics of the Galois cover
$$
(-64)+2(78+78) + 4(-88) = -104.
$$
The singularities on the Galois cover are the preimages of the $10$ nodes
of $Q_+$ and $10$ nodes of $Q_-$, which gives $40$ three-dimensional ODPs.
Each of them contributes an extra $1$ to stringy Euler characteristics since their
crepant resolution locally involves replacing a point with a $2$-sphere.
So we get the stringy Euler characteristics of $-104+40=-64$. As expected, this
matches the Euler characteristics of the quotient of the complete intersection
of type $(2,2,2,2)$ in $\Cc\Pp^7$ by a free involution.

\begin{remark}
We find it fascinating that the free quotient on the complete intersection side
leads to a double cover on the Clifford double mirror side by means of enlarging the center
of the corresponding sheaf of algebras. This phenomenon depends crucially on the parity
of the number of quadrics.
We will see later in Section \ref{222tau} that in the case of Enriques surfaces realized
as quotients of $(2,2,2)$ complete intersections in $\Cc\Pp^5$ by a free involution,
 the center will be generically just the field of functions on $\Cc\Pp^2$.
\end{remark}

\section{Derived categories of Clifford double mirrors.}\label{sectionderived}
The goal of this section is to provide a proof that the derived
category of the sheaf of Clifford algebras over a toric DM stack
constructed in the previous section is equivalent to the $\hat
G$-equivariant derived category of matrix factorizations
considered in Section \ref{philosophy}. This immediately implies
that in the case when different large K\"ahler limits give complete
intersection and a sheaf of Clifford algebras, there is an
equivalence of the corresponding derived categories. This is our
main result Theorem \ref{dmthm}.

\medskip
Recall that we have a
reflexive Gorenstein cone $K \subset {M}_\Rr$ with the dual
cone $K^\vee \subset{N}_\Rr$. Suppose there exists a decomposition
\[
\deg^\vee = \frac{1}{2}(s_1 + \cdots + s_{2k}),
\quad s_j \in K^\vee_{(1)},
\] and a regular triangulation $\Sigma$ of $K_{(1)}^\vee$ such that
every maximum dimensional cone of $\Sigma$ contains all $\{s_j\}$.
For a fixed generic coefficient function
\[
c:K_{(1)}\to\Cc
\]
we define a noncommutative variety/stack $  (\Ss, \Bb_0)$, as in Section \ref{defcliff}.
Recall that $\Ss=[U_{\overline \Sigma}/{\overline G}]$ is the toric DM stack associated to the action of the group
\[
\overline G = \hat G/H
\]
on an open set $U_\Sigma$ of $\Cc^{K^\vee_{(1)}-\{s_1,\ldots,s_{2k}\}}$.

\medskip

Let $D(\Ss, \Bb_0)$ be the bounded derived
category of coherent sheaves on $\Ss$ which are also $\Bb_0$-modules.
Let us also consider the derived category
$$
D_B(K,c;\Sigma)
$$
which is the category of factorizations of the potential $C$ on
$[U_\Sigma/G]$ defined by $c$. It is obtained from the category of
pairs of $\hat G$-equivariant sheaves $\Ff_1,\Ff_2$ on $U_\Sigma$,
with the maps $f:\Ff_1\to\Ff_2,g:\Ff_2\to \Ff_1\otimes \Ll$ whose
composition is multiplication by $C$.

\medskip
The main result of this section is the following.
\begin{theorem}\label{theorem: derived equivalence}
Under the flatness assumption of Remark \ref{flatornot}
there exists a derived equivalence
\[D(\Ss, \Bb_0)\cong D_B(K,c;\Sigma).
\]
\end{theorem}

\begin{proof}
The argument uses the intermediate category $D^b(\Xx)$, which contains both of the triangulated
categories in question as left and right orthogonal complements to an admissible subcategory.
We will use the results of \cite{Kuz08,BDFIK14}.

\medskip
In order to construct $\Xx$ and $D^b(\Xx)$,
recall that we have
$$
U_\Sigma = U_{\overline \Sigma}\times \Cc^{2k}.
$$
There is a group $\hat G$ acting on $U_\Sigma$. There is a subgroup
$H$ of $\hat G$ which acts by scaling the coordinates of $\Cc^{2k}$.
Moreover, the inclusion $H\subseteq \hat G$ 
splits by Remark \ref{splitHG}. Thus we will now consider
$$
\hat G = H\times {\overline G}
$$
where $\overline G = \hat G/H$ and have $\overline G$ act
 $U_\Sigma$ as well.

\medskip
Consider
 $\Cc\Pp^{2k-1}$ bundle $U_{\overline\Sigma}\times \Cc\Pp^{2k-1}$ over $U_{\overline \Sigma}$ given by
$$
U_{\overline\Sigma} \times ( \Cc^{2k}-\{{\bf 0}\})/H.
$$
The coefficient function $c:K_{(1)}\to \Cc$ gives rise to a quadric
fibration over $U_{\overline \Sigma}$ in the sense of \cite[Section
3]{Kuz08}. Specifically, the polynomial
$$
C({\bf z}) = \sum_{m\in K_{(1)}} c(m) \prod_{n\in K^\vee_{(1)}} z(n)^{\la m,n\ra}
$$
has total degree $2$ in the variables $z(s_1),\ldots,z(s_{2k})$, because
$$
\sum_{n=s_i,i=1,\ldots,2k} \la m, n \ra = \la m, \sum_{i=1}^{2k} s_i \ra = \la m, 2\deg^\vee\ra =2.
$$
While our situation is slightly more general, since we are interested in working equivariantly
with respect to the group ${\overline G}$ (or alternatively work over a DM stack base rather than
a scheme base), we can still use the framework of Kuznetsov. However, we do need to assume
that the fibration is flat, see Remark \ref{flatornot} and Section \ref{noflatness}.

\medskip
We denote by $\Xx$ the DM quotient substack $[\{C=0\}/H\times \overline
G]$ of the DM stack $[U_{\overline \Sigma} \times
(\Cc^{2k}-\{\mathbf{0}\})/H \times \overline G]$. Note that the
action of $\overline G$ descends to the action on
$U_{\overline\Sigma}\times \Cc\Pp^{2k-1}$ and is compatible with the
action on the base $U_{\overline\Sigma}$. Thus we have a quadric
fibration
$$
\pi:\Xx\to \Ss.
$$
Our definition of the sheaf of even Clifford algebras
$(\Ss,\Bb_0)$ matches that of \cite[Section 3]{Kuz08}.
We now wish to apply a slight generalization of \cite[Theorem 4.2]{Kuz08} to the equivariant
setting which states that $D^b(\Xx)$ admits a semiorthogonal decomposition
\begin{equation}\label{XtoS}
D^b(\Xx) = \big\la  D^b(\Ss,\Bb_0), \pi^*D^b(\Ss)\otimes O_{\Xx/\Ss}(1),
\ldots,
\pi^*D^b(\Ss)\otimes O_{\Xx/\Ss}(2k-2)
 \big\ra
\end{equation}
in the sense that a left orthogonal to the category generated by
$\pi^*D^b(\Ss)\otimes O_{\Xx/\Ss}(i), i=1,\ldots,{2k-2}$
 is equivalent to $D^b(\Ss,\Bb_0)$. Kuznetsov's
arguments apply in our slightly more general situation since his
construction is functorial.

\medskip
To compare $D^b(\Xx)$ with the category
$D_B(K,c;\Sigma)$ we need to consider a relative and equivariant version of the main
theorem of Orlov \cite[Theorem 16]{Orl09}. A relative (but not equivariant version) of this theorem
is proved in \cite{BDFIK14}, and their proof (based on \cite[Theorem 3.5.2]{BFK12}) works in the equivariant setting.
We thank Matt Ballard for providing this reference.

\medskip
The generalization of Orlov's theorem in question is provided by \cite[Proposition 3.3]{BDFIK14}
which states that
\begin{equation}\label{theotherside}
D^b(\Xx) = \big\la \pi^*D^b(\Ss)\otimes O_{\Xx/\Ss}(1),
\ldots,
\pi^*D^b(\Ss)\otimes O_{\Xx/\Ss}(2k-2),\Phi_+ \mathcal C
\big\ra
\end{equation}
where $\Phi_+$ is a certain fully-faithful functor and the category $\mathcal C$ is given
in \cite{BDFIK14} as $D^b(coh[Q_-/(\Cc^*)^2],w)$. In the equivariant setting, this would be
$$
D^b(coh[U_{\Sigma}\times \Cc^*/\hat G\times \Cc^*],uC({\bf z}))
$$
where $u$ is the coordinate on the extra $\Cc^*$. It is then easily seen to be equivalent to
$$
D^b(coh[U_{\Sigma}/\hat G],C({\bf z}) )= D_B(K,c;\Sigma).
$$
Therefore, \eqref{theotherside} and \eqref{XtoS}, together with the standard statement that left and right orthogonal complements
of an admissible subcategory are equivalent,  finish the argument.
\end{proof}

\begin{remark}
It would be much more pleasant to have a direct construction of equivalence between the categories in question. It was suggested to us by
Nick Addington that one should be able to go from a module $M$ over the Clifford algebra $\Bb_0$ to a matrix factorization by
looking at $M\to M\otimes_{\Bb_0}\Bb_1\to M$, where $\Bb_1$ is the odd part of the Clifford algebra. However, we do not have a proof
that this induces an equivalence of triangulated categories. We also do not know
whether this equivalence would coincide with the one  given by  the admittedly roundabout arguments above.
\end{remark}

As the consequence of Theorem \ref{theorem: derived equivalence} we get
the derived equivalence of double mirrors.
\begin{theorem}\label{dmthm}
Suppose that a complete intersection $\Xx$ and a Clifford
noncommutative variety $\Yy$ are given by different decompositions of the
degree element $\deg^\vee$ of a reflexive Gorenstein cone $K^\vee$
and the appropriate regular simplicial fans in $K^\vee$. Then the
bounded derived categories of $\Xx$ and $\Yy$ are equivalent,
provided the centrality and the flatness assumptions on $\Yy$ hold.
\end{theorem}

\begin{proof}
By Theorem \ref{theorem: derived equivalence} and Theorem \ref{isikthm}
 the derived categories in question are equivalent to two derived categories of factorizations
 defined by different fans in $K^\vee$. These are then equivalent by
Theorem \ref{samecat}.
\end{proof}

\section{Generalization to Clifford algebras over complete intersections.}\label{general}
The goal of this section is to indicate a generalization of the Clifford algebra construction
which would encompass both complete intersections and Clifford algebras over toric bases.

\medskip
As always, we consider a pair of reflexive Gorenstein cones $K$ and $K^\vee$
in lattices $M$ and $N$ respectively. Let us denote the degree elements by $\deg\in K$
and $\deg^\vee\in K^\vee$  and introduce the index $k=\la \deg,\deg^\vee\ra$.
In addition, we consider generic coefficient function
$$
c:K_{(1)}\to \Cc.
$$

\medskip
As explained in Section \ref{section2}, Calabi-Yau complete intersections in toric varieties appear
as a consequence of a decomposition of $\deg^\vee$ into a linear combination of elements
of $K^\vee_{(1)}$ with coefficients $1$. As we saw in Section \ref{defcliff} the Clifford algebras over
toric bases arise in the context of decomposition of $\deg^\vee$ into a linear combination with
coefficients $\frac 12$. We will now consider a more general case where both types of coefficients
may appear.

\medskip
Suppose that we have
$$
\deg^\vee =  \frac{1}{2}(s_1 + \cdots + s_{2r}) + t_1 + \cdots
+ t_{k-r}
$$
for some $0\leq r\leq k$. The $(k+r)$ elements $s_i$ and $t_j$ are supposed to be
linearly independent. In addition, there should exist (and be chosen) a
regular simplicial fan $\Sigma$ with support  $K^\vee$
such that the following centrality condition holds.
\begin{equation}\label{central3}
\emph{All maximum dimensional cones of $\Sigma$ contain
$\{s_i,t_j\}$ as ray generators.}
\end{equation}

\begin{remark}
The motivation for the above condition is our philosophy of large  K\"ahler limits of the
families of $N=(2,2)$ SCFTs. There is a large class of such limits given by different regular
simplicial cones $\Sigma$. We are interested in the situation where maximum cones of
$\Sigma$ contain $\deg^\vee$. While other limits may be of interest as well, they are likely
to lead to more complicated geometric descriptions of the triangulated category of boundary
conditions. Given such  fan $\Sigma$ it is natural to try to describe the
minimum cone that $\deg^\vee$ lies in, which is the intersection of all the maximum cones
of $\Sigma$. The element $\deg^\vee$ is a positive rational combination of the generators
of this cone, and in this paper we consider the case when these coefficients are $\frac 12$
or $1$.
\end{remark}

\begin{remark}
The case $r=k$ was considered in Section \ref{defcliff} and led to (sheaves of) Clifford algebras
over toric bases. The case $r=0$ is the usual complete intersection case reviewed in Section \ref{section2}.
\end{remark}

As in Section \ref{defcliff} we consider the open subset $U_\Sigma$ of $\Cc^{K^\vee_{(1)}}$
of functions
$$
{\bf z}:K^\vee_{(1)}\to \Cc
$$
such that the preimage of $0$ is a subset of $\Sigma$. Similarly, we consider the
subset
$$
U_{\overline\Sigma} \subset \Cc^{K^\vee_{(1)}-\{s_1,\ldots,s_{2r},t_1,\ldots,t_{k-r}\}}
$$
that corresponds to the stacky fan $\overline\Sigma$ for the group
$$
\overline N  = N /\Zz s_1 + \cdots + \Zz s_{2r}+ \Zz \deg^\vee+ \Zz
t_1+\cdots+ \Zz t_{k-r}. $$
We have
$$
U_\Sigma = U_{\overline\Sigma} \times \Cc^{2r}\times \Cc^{k-r}
$$
where the last coordinates correspond to values of ${\bf z}$ at $s_i$ and $t_i$.

\medskip
There is a group $\hat G$ defined by
$$
\hat G:=\{ { \bf \lambda}:K^\vee_{(1)}\to \Cc^*\Big\vert
\prod_{n\in K^\vee_{(1)}} \lambda(n)^{\la m,n\ra} = 1,~{\rm for~all~}m\in {\rm Ann}(\deg^\vee)\}
$$
as in Section \ref{defcliff}.
The analog of the subgroup $H$ of $\hat G$ which will be denoted by the same name is
given by $\Cc^*$ with
$$
\lambda(s_i) = t, \lambda(t_i) = t^2, \lambda(v)=1,{\rm ~for~all~}v\in K^\vee_{(1)}-\{s_1,\ldots,s_{2r},
t_1,\ldots,t_{k-r}\}.
$$
As in Section \ref{defcliff}, we see that the toric DM stack that corresponds to $(\overline N,\overline\Sigma)$ can be realized as the quotient of $U_{\overline\Sigma}$ by $\overline G= \hat G/H$.

\medskip
Note that the homogeneous polynomial
$$ C({\bf z})=\sum_{m\in K_{(1)} }c(m)  \prod_{n\in
K^\vee_{(1)}} z(n)^{\la m,n\ra}
$$
has total degree $2$ with respect to $H$. It has terms linear in
${\bf z}(t_i)$ and terms that have no ${\bf z}(t_i)$ but are
quadratic in ${\bf z}(s_i)$. We define the quadratic term by
$$ C_2({\bf z}) = \sum_{m\in K_{(1)} \cap {\rm
Ann}(t_1,\ldots,t_{k-r})} c(m)\prod_{n\in  K^\vee_{(1)}}
z(n)^{\la m,n\ra}.
$$
We use the linear terms to define the complete
intersection $Y\subset U_{\overline\Sigma}$ given by
$$ Y=\bigcap_{i=1}^{k-r} \Big\{\sum_{m\in K_{(1)}, \la m,
t_i\ra =1} c(m) \prod_{n\in  K^\vee_{(1)}-\{s_1,\ldots,s_{2r},
t_1,\ldots,t_{k-r}\}}z(n)^{\la m,n\ra}=0\Big\}
$$
where this intersection may assumed to be transversal if
the coefficient function $c$ is general. We then define the sheaf of
Clifford algebras  $\Bb_0$ on $\Ss=[Y/\overline G]$ as the pullback
of the sheaf of Clifford algebras on $[U_{\overline\Sigma}/\overline
G]$ under the natural inclusion. The aforementioned sheaf of
Clifford algebras is defined by using the quadratic part $C_2({\bf
z})$  of $C({\bf z})$. We will formulate this definition along the
lines of Remark \ref{explicitdefcliff}.

\begin{definition}\label{def.cliff.gen}
The category of coherent sheaves on $(\Ss,\Bb_0)$ is defined
as the category of $\overline G$-equivariant sheaves over the even part of the
(locally constant) sheaf of Clifford algebras over the (reduced) scheme
$Y\subseteq U_{\overline\Sigma}$ given by
$$
\Big(\Oo_{Y}\{y_1,\ldots,y_{2r}\}/\la (\sum_{i=1}^{2r} z_iy_i)^2+C_2({\bf z}),~~{\rm for~all~}z_1,\ldots,z_{2r} \ra\Big)_{even}
$$
where $y_1,\ldots, y_{2r}$ are free noncommuting variables.
\end{definition}

\begin{remark}
The subscript \emph{even} refers to the parity of the number of
$y_i$. Here the action of ${\overline G}=\hat G/H$ on
$y_1,\ldots,y_{2k}$ is defined as follows. We have the group $\hat
G$ act by scaling on all $z(n)$, in
particular on $z_1,\ldots,z_{2r}$. We will define its action on
products of even number of $y_i$ as follows. For ${\bf \lambda}\in
\hat G$ such that its image in $\hat G/G = \Cc^*$ is
$\varphi(\lambda)$ we define
$$
\lambda(y_iy_j) = \lambda_i^{-1}\lambda_j^{-1}\varphi(\lambda) y_iy_j
$$
to be the inverse of the action on corresponding $z_i$ twisted by
$\varphi(\lambda)$. This ensures that
$\lambda\Big((\sum_{i=1}^{2r}z_iy_i)^2\Big) = \varphi(\lambda)
(\sum_{i=1}^{2r}z_iy_i)^2$. On the other hand, $C_2({\bf z})$ is
also semi-invariant with respect to $\hat G$
with the character $\varphi(\lambda)$. Thus the ideal in Definition \ref{def.cliff.gen}
is preserved under
$\hat G$. Note that an element $t\in H=\Cc^*$ acts trivially on
 $\Oo_{U_{\overline\Sigma}}$ and on $y_iy_j$, since $\lambda_i=\lambda_j=t$
and
$\varphi(\lambda)=t^2$. Thus the action of $\hat G$ descends to the action
of $\hat G/H={\overline G}$.
\end{remark}

It is now reasonable to conjecture that these
 more sophisticated Clifford limits give the same triangulated
categories as any other limits we have considered.
\begin{conjecture}
Under the centrality and appropriate flatness assumptions, the bounded derived category of coherent sheaves on $(\Ss,\Bb_0)$ is equivalent to the category $D_B(K,c,\Sigma)$.
\end{conjecture}

\section{Combinatorics of Clifford decompositions.}\label{sectioncomb}
In the first part, we explore the combinatorics
associated to a decomposition
\begin{equation}\label{decgen}
\deg^\vee = \frac{1}{2}(s_1 + \cdots + s_{2r}) + t_1 + \cdots
+ t_{k-r}.
\end{equation}
Large parts of this section can be read independently from the more technical parts of the paper
that discuss derived equivalences and Clifford algebras.

\medskip
Just as before, we have a pair of reflexive Gorenstein cones $K
\subset M_\Rr, K^\vee \subset N_\Rr$ of index $k$. Suppose there
exists a decomposition \eqref{decgen} where
$s_i, t_j \in K^\vee_{(1)}$ are linearly independent lattice
elements. We define a lattice
\[\overline{N}_{free} = N/{(N \cap (\sum_{i=1}^{2r} \Rr
\cdot s_i + \sum_{j=1}^{k-r}\Rr \cdot t_j))}.\]
Notice that $\overline{N}_{free}$ is the quotient of
$\overline N$ defined in Section \ref{section2} by its torsion subgroup. Induced from the
pairing between $M$ and $N$, the dual lattice of $\overline N_{free}$ is
\[
\begin{split}
\overline{M}&=\Ann(s_1, \cdots, s_{2r}, t_1,\cdots,t_{k-r})\\
&=\{m \in M \mid \langle m ,s_i \rangle =\langle m ,t_j
\rangle =0~ \forall~ 1 \leq i \leq 2r, 1 \leq j \leq k-r\}.
\end{split}
\]We also use $\langle -, - \rangle$ to denote the induced perfect pairing $\overline{M} \times {\overline{N}_{free}} \to \Zz$.

\medskip

Let $\Theta$ be the image of the convex hull of $K_{(1)}^\vee$ in
$\overline{N}$. It is a lattice polytope by the definition of
Gorenstein cone, and it contains the origin as an interior element.
Indeed, any linear function on ${\overline{N}_{free}}$
 lifts to a linear function on $N$
which is zero on $\deg^\vee$. Since $\deg^\vee$ is in the interior
of $K^\vee$, this  function takes positive and negative values on
some rays of $K^\vee$. Therefore, any linear function on
${\overline{N}_{free}}$ takes positive and negative values
on $\Theta$, which implies that the origin lies in the interior of
$\Theta$.

\medskip
We introduce a polytope in $\overline M_\Rr$ defined by
\begin{equation}\label{eq: T}
T: = \{x \in K \mid \langle x,s_i\rangle = \langle x, t_j
\rangle = 1~ \forall~ i, j \} -\deg .
\end{equation}
We should point out that $T$ may not be a lattice
polytope, however, we will show that its dual polytope is the
lattice polytope $\Theta$.

\begin{lemma}
The polytope $T$ contains origin in its interior. There holds
$$
T^\vee = \{y\in{\overline N_\Rr},{\rm such~that~} \la T,y\ra\geq -1\}=\Theta.
$$
\end{lemma}

\begin{proof}
Let us investigate the dual of $\Theta$. By definition of the dual
polytope and of $\Theta$, the dual $\Theta^\vee$ is
 a subset of $\overline M_\Rr$ which consists of
$x$ such that $\la x, K^\vee_{(1)}\ra \geq -1$. In other words,
this is a subset of $M_\Rr$ such that
$\la x,K^\vee_{(1)}\ra \geq -1$ and
$\la x,s_i\ra= \la x,t_j\ra = 0$ for all $i$ and $j$.

\medskip
Equivalently, $x+\deg$ can be characterized by
$$
\la x+\deg, K^\vee_{(1)}\ra \geq 0,~
\la x+\deg,s_i\ra= \la x+\deg,t_j\ra = 1,~{\rm for~all~}i,j
$$
which is precisely the definition of $T+\deg$. This
shows that $T=\Theta^\vee$, which implies $T^\vee=\Theta$ and ${\bf 0}\in T^\circ$.
\end{proof}

\medskip
We will now define the following three sets of polytopes:
\[
A_i = \Conv\{x \in K_{(1)} \mid \langle x, t_i \rangle = 1\}, ~1 \leq i \leq k-r
\] and
\[
\begin{split}
T_{i,i} &= \Conv\{x \in K_{(1)} \mid \langle x, s_i \rangle= 2\}, ~1 \leq i \leq 2r\\
T_{i,j} &=   \Conv\{x \in K_{(1)} \mid \langle x, s_i \rangle= \langle x, s_j \rangle  =1\},~ 1 \leq i, j \leq 2r {\rm~and ~}i \neq j.
\end{split}
\] Because a lattice point $x \in K_{(1)}$ pairs with $\deg^\vee$ at
$1$, $x$ must uniquely lie in one of $A_i, T_{i,i} $ and $T_{i,j}$. In
other words, the primitive elements of rays of $K$ can be classified
according to the polytopes $A_i, T_{i,i} $ or $T_{i,j}$ they lie in.
Note that elements of these sets pair by $0$ with the other $s$ and $t$ points.

\medskip

We now define
\[\DT : = \Conv \left(\bigcup_{\sigma \in S_{2r}} \sum_{ 1
\leq i \leq 2r} T_{i, \sigma(i)}\right)\] be a lattice polytope,
where $S_{2r}$ denotes the $2r$-symmetric group. A priori, $\DT$
could be empty, but it is a consequence of Theorem \ref{dual
polytope} that this can never happen in our setting. Let
\[S: = \sum_{i = 1}^{2r} A_i + \frac{1}{2} \DT -\deg\] be a polytope.
By pairing with $s_i,t_j$, one can verify directly that $S \subset
\overline M$, and we will show in Theorem \ref{dual polytope} that it
coincides with $T$, that is $S = \Theta^\vee$.
In order to prove it we need to first recall the following so called Birkhoff--von
Neumann theorem.

\begin{definition}
A permutation matrix is the matrix with exactly one entry $1$ in
each row and column and $0$ elsewhere.
\end{definition}

\begin{lemma}(Birkhoff--von Neumann theorem)\label{BVN theorem}
Suppose $B_n$ is an $n\times n$ matrix whose entries are non-negative
real numbers and whose rows and columns each add up to $1$. Then
$B_n$ lies in the convex hull of the set of $n\times n$ permutation
matrices.
\end{lemma}

We now state the following key result:

\begin{theorem}\label{dual polytope}
The polytope \[S: =  \sum_{i = 1}^{k-r} A_i + \frac{1}{2}
\Conv \left(\bigcup_{\sigma \in S_{2r}} \sum_{ 1 \leq i \leq 2r}
T_{i, \sigma(i)}\right) -\deg\] is equal to  $T$. In particular, its
dual polytope is $\Theta$.
\end{theorem}

\begin{proof}
By definition of $T$ (see equation \eqref{eq: T}), we only need to show
\[\begin{split}
&\sum_{i = 1}^{k-r} A_i + \frac{1}{2} \Conv \left(\bigcup_{\sigma \in S_{2r}} \sum_{ 1 \leq i \leq 2r} T_{i, \sigma(i)}\right) \\
&= \{x \in K \mid \langle x,s_i\rangle = \langle x, t_j \rangle = 1~
\forall~ i, j \}.
\end{split}
\] The inclusion $\subseteq$ is obtained by definition, thus we only need to show the inverse inclusion.

\medskip

Suppose $x \in \{x \in K \mid \langle x,s_i\rangle = \langle x, t_j
\rangle = 1~ \forall~ i, j \}$, then \[x = \sum_{v\in K_{(1)}} \lambda_v v,
\quad \lambda_v \geq 0.\] Because the generators of $K$
can be classified according to the polytopes $A_i, T_{i,i}$ or
$T_{i,j}$ they lie in, we can rewrite the summation as
\[
x = \sum_i \sum_{v \in A_{i}}\lambda_v v + \sum_{i,j}\sum_{v \in
T_{i,j}} \lambda_v v.
\] Notice that when $i \neq j$, we have $T_{i,j}=T_{j,i}$.

\medskip

Let
\begin{equation}\label{equation: x-t and x-s}
\begin{aligned}
&x_t = \sum_i \sum_{v \in A_{i}}\lambda_v v, \\
&x_s = \sum_{i,j}\sum_{v \in T_{i,j}} \lambda_v v =
\sum_{i,j}b_{i,j}y_{i,j},
\end{aligned}
\end{equation}
 with $y_{i,j} \in T_{i,j}$. Besides, we require $y_{i,j}=y_{j,i}$
and  $b_{i,j}=b_{j,i}$.

\medskip
We have $x = x_t + x_s$, moreover, by pairing with $t_i$, we
have \[ 1 = \langle x, t_i \rangle = \langle x_t, t_i \rangle
=\sum_{v \in A_{i}} \lambda_v.
\] This implies that $\sum_{v \in A_{i}} \lambda_v v \in A_{i}$,
thus $x_t \in \sum_{i=1}^{k-r} A_{i}$. Hence, all we need to show is
$x_s \in \frac{1}{2}\DT$.

\medskip

By pairing with $s_j$, we have
\begin{equation}\label{eq: 1}
1 = \langle x , s_j \rangle = \langle x_s, s_j \rangle = 2b_{j,j} +
\sum_{\substack{1 \leq i \leq 2s \\ i\neq j}} b_{i,j} + \sum_{\substack{1 \leq i \leq 2s \\
i\neq j}}b_{j,i} = 2 \sum_{1 \leq i \leq 2s}b_{i,j}.
\end{equation}

Let $B_{2r} = (b_{i,j})_{i,j}$ be a $2r \times 2r$-symmetric matrix.
Then by $\eqref{eq: 1}$, we have \[ \sum_{1 \leq i \leq 2r}b_{i,j}
=\sum_{1 \leq j \leq 2r}b_{i,j} = \frac{1}{2}.
\] According to the Birkhoff--von Neumann theorem (Lemma~\ref{BVN
theorem}), $B_{2r}$ is a convex linear combination of permutation
matrices.

\medskip

There is a 1-1 correspondence between elements of $2r$-symmetric
group $S_{2r}$ and $2r \times 2r$ permutation matrices under the
map\[ \sigma \mapsto (\delta_{j,\sigma{(i)}})_{i,j}, \] where
$\delta_{j,\sigma{(i)}}$ is the Kronecker delta. Because of this,
there exist $r_\sigma \geq 0$, such that
\[b_{i,j} = \sum_{\sigma \in S_{2r}}r_\sigma \delta_{j,\sigma{(i)}}.
\] For fixed $j$, because $\sum_{1 \leq i \leq 2r}b_{i,j} = \frac{1}{2}$,
we have
\begin{equation}\label{eq: 2}
\sum_{1 \leq i \leq 2r} \sum_{\sigma \in S_{2r}} r_\sigma
\delta_{j,\sigma{(i)}} = \sum_{\sigma \in S_{2r}}
\sum_{\substack{j=\sigma{(i)}\\1 \leq i \leq 2r}} r_\sigma  = \sum_{\sigma \in S_{2r}} r_\sigma =
\frac{1}{2}.
\end{equation} We  use the equation \eqref{equation: x-t and x-s} to get
\[\begin{split}
x_s &= \sum_{i,j}b_{i,j}y_{i,j}=\sum_{i,j}(\sum_{\sigma \in
S_{2r}}r_\sigma \delta_{j,\sigma{(i)}})y_{i,j}\\
&= \sum_{\sigma \in S_{2r}}r_\sigma
(\sum_{\substack{j=\sigma{(i)}\\i,j}}y_{i,j}) = \sum_{\sigma \in
S_{2r}}r_\sigma (\sum_{1 \leq i \leq 2r}y_{i,\sigma(i)})
\end{split}
\] Combining this with $\eqref{eq: 2}$, we have \[x_s \in \frac{1}{2}\Conv
\left(\bigcup_{\sigma \in S_{2r}} \sum_{ 1 \leq i \leq 2s} T_{i,
\sigma(i)}\right),\] which finishes the proof.
\end{proof}

\begin{corollary}
The polytope $2T$ has lattice vertices.
\end{corollary}

\begin{proof}
Indeed, by Theorem \ref{dual polytope} we see that
\[2T=2S= \sum_{i = 1}^{k-r} 2A_i + \Conv
\left(\bigcup_{\sigma \in S_{2r}} \sum_{ 1 \leq i \leq 2r} T_{i,
\sigma(i)}\right) -2\deg\]
is a lattice polytope, since Minkowski sums, convex hulls
and lattice shifts of lattice polytopes are again lattice polytopes.
\end{proof}

\begin{remark}
The
geometric meaning of the above corollary is that the toric variety
$\Pp$ defined by $T$ is Fano $\Qq$-Gorenstein with $(-2K_{\Pp})$ an
ample Cartier divisor.
\end{remark}

The data of the coefficient function
$$
c:K_{(1)}\to \Cc
$$
may be equivalently encoded as a collection of $(k-r)$
 Laurent polynomials $f_i$ with Newton polytopes
$A_i$ and an $(2r)\times(2r)$ symmetric matrix $R$ of Laurent
polynomials with Newton polytopes $T_{i,j}$. These data allow us to
consider a double cover of the complete intersection of $\{f_i=0\}$
in the toric variety defined by $T$, ramified
over $\{\det(R)=0 \}$. The resulting variety is
Calabi-Yau. However, it has singularities other than the ones coming
from the ambient toric varieties due to the loci of corank two or
higher of the matrix $R$. One can view the Clifford variety of
Section \ref{defcliff} as a noncommutative crepant resolution of
this double cover (with a certain Brauer class).

\medskip
In many (though not all) cases we have the important technical centrality assumption \eqref{central3}
that there exists a regular simplicial fan $\Sigma$ all of whose maximum dimensional
cones contain all $s_i$ and all $t_j$.

\medskip
As a consequence of the assumption \eqref{central3}, there is a
(stacky) fan $\overline\Sigma$ on the quotient
$$ \overline
N = N /
\Zz s_1 + \cdots + \Zz s_{2r}+
\Zz \deg^\vee +
 \Zz t_1+\cdots+ \Zz t_{k-r}$$
obtained by removing $s_i$ and $t_j$ from the sets of
$\Sigma$. Then we have natural line bundles
$\Ll'_{1},\ldots,\Ll'_{2r}$ and $\Ll_{1},\ldots,\Ll_{k-r}$
on the stack $\Pp_{\overline\Sigma}$ with the property that
$$
\bigotimes_{i=1}^{2r} \Ll'_i \otimes \bigotimes_{j=1}^{k-r} \Ll^{\otimes 2}_{j}
$$
is the square of the anticanonical bundle on $\Pp_{\overline\Sigma}$, as considered
in Section \ref{general}.

\begin{remark}
While $\Ll_i$ are pullbacks of the Cartier divisors from the coarse moduli space,
the same can not be guaranteed for $\Ll'_j$.
The data of the coefficient function
$c:K_{(1)}\to\Cc$ amounts to a choice of $(k-r)$ sections
$f_1,\ldots, f_{k-r}$ of $\Ll'_1,\ldots,\Ll'_{k-r}$ respectively.
\end{remark}

\medskip
Then we can view $A_{i}$ as the Newton polytopes that support the
sections of $\Ll_i$, and the polytopes $T_{i,j}$
 are the ones that support sections of $\Ll'_i\otimes \Ll'_j$ under
appropriate linearizations.

\begin{remark}
In the absence of the centrality condition \eqref{central3}, we still obtain a singular Calabi-Yau variety
which is a double cover of the complete intersection in a toric $\Qq$-Gorenstein Fano variety
given by the polytope $\Theta$. It would be interesting to see if one can find noncommutative
resolutions of its singularities and whether there is still a derived equivalence statement. A priori,
one no longer has the vector bundle structure on $U_\Sigma$, which prevents one from directly
using the work of Kuznetsov.
\end{remark}

\begin{remark}
It would be interesting to investigate possible torsion in
$$
\overline N = N / \Zz s_1+\cdots+ \Zz s_{2r}+ \Zz \deg^\vee+\Zz t_1 + \cdots + \Zz
t_{k-r}
$$
If the sets $T_{ii}$ are nonempty, then we one can show that this is
at most $2$-torsion. However, we don't
know if $T_{ii}\neq \emptyset$ holds in general.
\end{remark}

\section{More examples.}\label{sectionex}
We will describe some examples of Clifford double mirrors  in the
literature as well as new examples. Each of these examples consists of a pair of Calabi-Yau
varieties, and the evidence for the double mirror property comes from equivalence
of derived categories.

\subsection{Example: $(2,2,2)$-complete intersections in $\Cc\Pp^5$.}
This example is given by Mukai in \cite{Muk88} (see Examples
(1.5)(1.6)(2.2)). Let $q_i, 0 \leq i \leq 2$ be quadratic equations
in $\Cc\Pp^5$, and $Q_i$ be the corresponding quadric hypersurfaces.
Suppose these $Q_i$ intersect transversally in $\Cc\Pp^5$, then their
complete intersection $X$ is a K3 surface. On the other hand, let
$A_i$ be the symmetric $6 \times 6$ matrix corresponding to the
quadratic form $q_i$. Then a quadric $\{a_0 q_0 + a_1 q_1 +a_2 q_2=0\}$ is
smooth if and only if the matrix $a_0 A_0 + a_1 A_1 +a_2 A_2$ is
regular. Hence, their singular members are parameterized by the
degree $6$ curve $D:=\{\det(a_0 A_0 + a_1 A_1 +a_2 A_2)=0\}$ in
$\Cc\Pp^2$ (where $a_0, a_1 ,a_2$ are variables). The double cover
ramified along $D$ is a K3 surface, and we denote it by $Y$.

\medskip

We further assume that every quadric containing $X$ is of rank $\geq
5$. Let $h \in H^2(X,\Zz)$ be the cohomology class of hyperplane
sections of $X$, then the moduli space of stable (with respect to
polarization $X \to \Cc\Pp^5$) rank $2$ vector bundle with $c_1 = h$
and $c_2 = 4$ is canonically isomorphic to $Y$. By  \cite[Section
5.5]{Cal00a}, there exists an $\alpha \in Br(Y)$ in the
Brauer group of $Y$ such that $D^b(X) \cong D^b(Y,\alpha)$.

\medskip

This example is a particular case of Kuznetsov's construction
for $k=4$ and can consequently be reconstructed by our method.
The reader can either follow the discussion of Section
\ref{2222}
or the description below along the lines of Section \ref{sectioncomb}.

\medskip
Let $\Delta$ be the polytope which is the convex hull of $\{e_1,
\ldots, e_5, e_6\}$, where $e_i, 1 \leq i \leq 5$ are standard basis
of $\Zz^5$, and $e_6 = -\sum_{i=1}^5 e_i$. The normal fan of
$\Delta$ is the fan of $\Cc\Pp^5$. Let $\Delta_j = \Conv\{e_{2j-1},
e_{2j}, \mathbf{0}\}~ \forall~ 1 \leq j \leq 3$. Then
$\{\Delta_i \mid 1 \leq i \leq 3\}$ is a nef-partition of
$\Delta$.

\medskip
We can define a reflexive Gorenstein cone associated to $\Delta_i$ by
\[K^\vee =: \{(a, b, c; a\Delta_1 + b \Delta_2 + c\Delta_3) \mid a, b, c \in  \Rr_{\geq 0}\}\subset  \Rr^8. \]  This is exactly how one can associate a reflexive Gorenstein cone to a nef-partition (see \cite{BB94}). The $(2,2,2)$-complete
intersection $X$ is just the complete intersection
defined by decomposition
$\deg^\vee = (1,0,0;\mathbf{0}) + (0,1,0;\mathbf{0})+
(0,0,1;\mathbf{0})$).

\medskip

On the other hand, $\deg^\vee$ can also be presented by
\[
\deg^\vee = \frac{1}{2}(s_1 + \cdots +s_6),
\] where
\[
\begin{split}
&s_1 = (1,0,0;e_1), \quad s_2 =(1,0,0;e_2)\\
&s_3 = (0,1,0;e_3), \quad s_4 =(0,1,0;e_4)\\
&s_5 = (0,0,1;e_5), \quad s_6 =(0,0,1;e_6).\\
\end{split}
\] Then $\Zz^8 \cap (\sum_{i=1}^6 \Rr  s_i) = \Zz  \deg^\vee
+ \sum_{i=1}^6 \Zz  s_i$, hence
\[\overline{N}_{free}=\overline{N} = \Zz^8/(\Zz \deg^\vee +
\sum_{i=1}^6 \Zz  s_i).\] Again, let $\Theta$ be the image of the
$K_{(1)}^\vee$ in $\overline{N} \cong \Zz^2$. Then $\Theta$ is a
convex hull of $\{v_1:=\overline{(1,0,0;\mathbf{0})},
v_2:=\overline{(0,1,0;\mathbf{0})},
v_3:=\overline{(0,0,1;\mathbf{0})}\}$, because all the other
vertices of $K^\vee_{(1)}$ is zero in $\overline{N}$. Moreover,
because $(1,0,0;\mathbf{0}) + (0,1,0;\mathbf{0}) +
(0,0,1;\mathbf{0}) = 2 \deg^\vee$, we have $v_1 + v_2 + v_3 =0$.
Thus, the normal fan of $\Theta$ is exactly the fan for $\Cc\Pp^2$.

\medskip

We can define $g =  \xx^{-2\deg}\det((g_{i,j})_{i,j}) $ as above,
where $g_{i,j}$ is the Laurent polynomial constructed from
$T_{i,j}$. By Theorem \ref{dual polytope}, $g$ is a global section
of $H^0(\Cc\Pp^2,
\Oo(-2K_{\Cc\Pp^2}))$, hence of degree $6$. The Calabi-Yau variety
$Y$ in the construction is exactly the double cover
ramified along the sextic ${D}: = \{g=0\}$.

\medskip
\subsection{Double mirrors of Enriques surfaces.}\label{222tau}
We recall that Enriques surfaces are quotients of certain K3
surfaces by a fixed point free involution. One of the many
constructions of these surfaces is provided by the quotients of
$(2,2,2)$ complete intersections in $\Cc\Pp^5$.
Specifically, we need to consider an action of an involution on
$V=\Cc^6$ that has trace $0$ and take a complete intersection of
three invariant quadrics on $\Pp V$. The involution fixes this
surface and acts freely on it. The resulting quotient surface is
Enriques.

\medskip
The corresponding Gorenstein cones are given as in Section
\ref{2222tau} as follows. We
have
\[
N=\Big((\bigoplus_{i=1}^6 \Zz s_i + \frac 12 \Zz (s_1+s_2+s_3))\oplus \bigoplus_{j=1}^3\Zz t_j\Big)/ \Zz(\sum_{i=1}^6s_i-2\sum_{j=1}^3t_j).
\]
The cone $K^\vee$  is the image of the nonnegative orthant.
The dual lattice $M$ is given by
\[
M=\{\sum_{i=1}^6 a_i s_i^\vee + \sum_{j=1}^3 b_j t_j^\vee \mid \sum_{i=1}^6a_i=2\sum_{j=1}^3b_j~{\rm and~} \sum_{i=1}^3a_i {\rm ~is~even}\}.
\]
The cone $K$ is the intersection of $M$ with the
nonnegative orthant.

\medskip
The usual Kuznetsov's double mirror is the sheaf of Clifford
algebras over $\Cc\Pp^2$ whose center is the double cover of
$\Cc\Pp^2$ ramified at the union of two elliptic curves $E_+$ and
$E_-$ which are written as determinants of symmetric $3\times 3$
matrices of linear forms. The action of the involution means that we
need to consider the corresponding Clifford algebra over the gerbe
$[\Cc\Pp^2/\Zz_2]$.

\medskip
As in Section \ref{2222tau}, we consider the semidirect product of
the Kuznetsov's sheaf of Clifford algebras  over
$\Cc\Pp^2$ and the group ring $\Cc[h]/\la h^2-1\ra$ of $\Zz_2$. Over
the generic point of $\Cc\Pp^2$, after diagonalization of the
quadratic forms, we get the even part of the quotient of the free
algebra over the field of rational functions $F$ on $\Cc\Pp^2$
$$
F\{y_1^+,y_2^+,y_3^+, y_1^-,y_2^-,y_3^-,h\}
$$
by the two-sided ideal generated by the relations
$$
h^2-1, hy_i^+ - y_i^+ h, hy_i^- + y_i^- h,
(y_i^+)^2+c_i^+,(y_i^-)^2+c_i^-,
y_i^+y_j^- + y_j^-y_i^+
$$
for all $i$ and $j$ and
$$
y_i^+y_j^+ + y_j^+y_i^+,y_i^-y_j^- + y_j^-y_i^-
$$
for $i\neq j$.
Again $c_i^\pm\in F$
may not be assumed to be $1$, since $F$ is not algebraically closed. In fact, up to
squares, the products $\prod_{i=1}^3c_i^\pm$ give equations of $E_\pm$.

\medskip
The calculation of the center of the above algebra is done analogously to Section \ref{2222tau}
but yields a different result. In fact, one simply gets $F$ as the center.

\begin{remark}\label{log}
Some more delicate preliminary calculations seem to indicate that the structure of the
double mirror of an  Enriques surface is that of (smooth) $\Zz_2\times \Zz_2$ root stack over $\Cc\Pp^2$, ramified over the union of two elliptic curves $E_+$ and $E_-$, presumably with a Brauer element.
The corresponding orbifold Euler characteristics calculation is as follows. The complement of
the union of two elliptic curves has $\chi=12$. The elliptic curves
contribute nothing, since the ages of the corresponding involutions are $\frac 12$ so the contribution
of the twisted sector cancels that of the untwisted sector. Similarly,  the contributions of the $9$
intersection points are zero, since the $4$ sectors cancel each other. Thus the Euler characteristics
of the root stack matches that of the Enriques surface. We thank Howard Nuer for pointing out a likely relationship
between these double mirrors and the construction of Enriques surfaces as logarithmic transformations of elliptic surfaces in \cite[p.599]{GH}.
\end{remark}

\subsection{Calabrese-Thomas' example.}

We state the construction of Calabrese and Thomas' first example in
\cite{CT14}.\footnote{The second example of \cite{CT14} falls into
the framework of Batyrev-Nill double mirrors from Section
\ref{section2}.} Let $V,W$ be complex vector spaces of dimension
$3$, and $\Pp(V \oplus W) = \Cc\Pp^5$. Let \[\pi: Z:= {\rm
Bl}_{\Pp(V)}(\Cc\Pp^5) \to \Cc\Pp^5\] be the blowup along $\Pp(V)$
with exceptional divisor $E$. Then one can show that
$\pi^*\Oo_{\Cc\Pp^5}(3)(-E)$ is a base point free line bundle. Let
$f_1, f_2$ be two general global sections of
$\pi^*\Oo_{\Cc\Pp^5}(3)(-E)$. Let \[X: = \{f_1 = f_2 = 0\} \subseteq
Z\] be their complete intersection. Since the anticanonical
bundle of $Z$ is equal to
$\pi^*\Oo_{\Cc\Pp^5}(6)(-2E)$, we see that $X$ is a Calabi-Yau
variety by the adjunction formula.

\medskip

Let $\rho: Z \to \Pp(W)$ be the projection from the plane $\Pp(V)
\subseteq \Cc\Pp^5$ to $\Pp(W)$, then one can show that the universal
hypersurface $\mathcal{H}:=\{x_1f_1 + x_2f_2 = 0\} \subseteq Z
\times \Cc\Pp^1$ is a quadric fibration over $Y = \Pp(W) \times \Cc\Pp^1$
under the morphism $\rho \times {\rm id}$. In particular, this quadric
fibration corresponds to an even Clifford algebra sheaf
$\Bb_0$.  As a consequence of the relative version of
homological projective duality, there is a derived equivalence
$D^b(X) \cong D^b(Y, \Bb_0).$

\medskip

One can modify the right hand side of the equivalence by considering
the relative Fano scheme of lines of $\rho \times id: \mathcal{H}
\to \Pp(W) \times \Cc\Pp^1$ (see~\cite{Kuz14}). Since $\rho \times {\rm id}$
is a quadric fibration, we can consider the loci $D'$ where the
rank of the quadratic form on the fibre is not of full rank (i.e.
$<4$). Then one can show that $D'$ is a $(6,4)$-bidegree divisor on
$Y = \Pp(W) \times \Cc\Pp^1 \cong \Cc\Pp^2 \times \Cc\Pp^1$. Let $F \to  \Pp(W)
\times \Cc\Pp^1$ be the relative Fano scheme of lines of $\rho \times
{\rm id}$, and $F \to  X' \to Y$ be its Stein
factorization, then $X' \to  Y$ is a double cover
ramified along $D'$. Finally, let $X'' \to X'$ be some small
resolution (in the analytic category), then there exists an $\alpha
\in Br(X'')$, such that \[D^b(X) \cong D^b(Y,
\Bb_0) \cong D^b(X'',\alpha).\]

\medskip

The reconstruction of this example is similar to the above cases, but a little bit
involved:

\medskip
We start by considering a $\Zz^5$ with the usual $\Cc\Pp^5$
fan on it, namely the one with
$e_1=(1,0,0,0,0),...,e_5=(0,0,0,0,1),e_6=(-1,...-1)$. Then the
blowup of $\Pp(V) \cong \Pp^2$ introduces an additional vertex
$e_0=(1,1,1,0,0)$. Let $\Delta$ be the convex hull of $\{e_0,
\cdots, e_6\}$, then it has a nef-partition with
\[
\Delta_1 = \Conv(\{{\bf 0},e_0,e_1,e_2,e_6\}), \quad \Delta_2 =
\Conv(\{{\bf 0},e_3,e_4,e_5\}).
\] Let $K^\vee \subseteq \Rr^7$ be the reflexive cone defined by
\[
K^\vee:= \{(a, b ; a\Delta_1 + b \Delta_2 \mid a, b \in \Rr_{\geq 0})\} \subseteq \Rr^7.
\]

\medskip

Again, $X$ is the complete intersection associated to $\deg^\vee =
(1,0;{\bf{0}}) + (0,1; {\bf{0}})$, and one can verify that the nef divisors
associated to $\Delta_i, i=1,2$ are both linearly equivalent to
$\pi^*\mathcal{O}_{\Cc\Pp^5}(3)(-E)$ in Calabrese-Thomas' example.

\medskip

Next, we write
\[
\deg^\vee = \frac{1}{2}(s_1 + s_2 + s_3 +s_4),
\] where
\[
s_1=(1,0;e_0), \quad s_2 = (1,0;e_6), \quad s_3 = (0,1;e_4), \quad s_4 = (0,1;e_5).
\] One can show that \[\Zz^7 \cap  \sum_{i=1}^4\Rr\cdot s_i  = \Zz \cdot \deg^\vee + \sum_{i=1}^4 \Zz \cdot s_i,\] hence
\[
\overline N_{free}=\overline{N} = \Zz^7/(\Zz \cdot
\deg^\vee + \sum_{i=1}^4 \Zz \cdot s_i).
\] Let $\Theta$ be the image of $K_{(1)}^\vee$ under the quotient map. Recall that $K_{(1)}^\vee$ has $9$ generators
\[
\begin{split}
&(1,0;\mathbf{0}), \quad (1,0;e_0),\quad (1,0;e_1),\quad (1,0;e_2),\quad (1,0;e_6)\\
&(0,1; \mathbf{0}),\quad (0,1; e_3),\quad (0,1;e_4),\quad (0,1;e_5).
\end{split}
\] After taking the quotient, $4$ of them disappear, and the other $5$ forms a fan of $Y = \Pp(W) \times \Cc\Pp^1 \cong \Cc\Pp^2 \times \Cc\Pp^1$. Specifically,
\[
\begin{split}
&(0,1;{\mathbf 0}) + (1,0;{\mathbf 0}) = \deg^\vee ~(= 0 \in \overline{N}),\\
&(1,0;e_1) + (1,0;e_2) + (0,1;e_3) = \deg^\vee + (0,1;e_0) ~(= 0 \in \overline{N}).
\end{split}
\]

By a computer search, there are $96$ lattice points in $K_{(1)}$
which fall into $10$ classes depending on the pairing with $s_i$. By
pairing with
$$\overline{(1,0;{\bf 0})}, \overline{(1,0; e_1)},
\overline{(1,0; e_2)} \in \overline{N},$$
we can find out the degree of
these classes in $\Cc[\overline{N}]$. It turns out that the
determinantal equation $g$ is exactly of bidegree $(6,4)$. Hence the
double cover of $Y$ ramified along $D'=\{g=0\}$
gives the Calabi-Yau variety $X'$ which coincides with the
construction in Calabrese-Thomas' example.

\begin{remark}
Calabrese and Thomas consider analytic small resolutions of the singular double
cover, similar to \cite{Add09}. In contrast, our construction gives a
noncommutative algebraic resolution
by a sheaf of Clifford algebras. We also remark that the flatness assumption is satisfied in
Calabrese-Thomas' example.
\end{remark}

\subsection{An example with $k=1$.}\label{subsection: a first non-trivial example}
As we mentioned before, even when the index $\langle \deg, \deg^\vee \rangle
$ is equal to $1$, there may exist double mirrors in our construction, in contrast
to the Batyrev's original construction~\cite{Bat94}. In the
following, we will give an example of such type in a $2$ dimensional
lattice. The resulting Calabi-Yau varieties are (elliptic) curves,
therefore the Brauer class has to be trivial.
This could be viewed as an almost trivial example of our
construction,  but it is informative to work it out in detail.

\medskip

We consider the $2$ dimensional reflexive polytope \[\Delta=\Conv\{(1,1),(1,-1),(-1,-1),(-1,1)\} ,\] whose dual polytope is
\[\Delta^\vee =
\Conv\{(1,0),(0,-1),(-1,0),(0,1)\}.\]

Then we have a pair of reflexive Gorenstein cones
\[
K = \{(a;a\cdot\Delta) \mid a \geq 0\} \subset M_\Rr,\quad K^\vee =
\{(b;b\cdot\Delta^\vee) \mid b \geq 0\} \subset N_\Rr.
\] The picture of these cones is given in Section \ref{subsection: Reflexive Gorenstein cones}. We can decompose $\deg^\vee$ in two different ways:
\[
\deg^\vee = (1;\mathbf{0}) = \frac{1}{2}(s_1 + s_2),
\] where $s_1 = (1,-1,0), s_2 = (1,1,0)$.

\medskip

Let $X$ be the elliptic curve defined by the decomposition
$\deg^\vee = (1;\mathbf{0})$. This is exactly the Batyrev-Borisov variety as we just take the complete intersection and do not need to take the double cover. It is a hypersurface associated to the
anticanonical divisor in the toric variety $\Cc\Pp^1 \times \Cc\Pp^1$. Hence, a global section $f$ of the
anticanonical divisor can be identified with $\{\sum a_v \xx^v \in
\Cc[M]\mid v \in \nabla\}$.
Since a smooth curve is uniquely determined up to birational
equivalence, $X$ is uniquely determined by its intersection with
$(\Cc^*)^2 \subset \Cc\Pp^1 \times \Cc\Pp^1$.
Let
\[\begin{split}
f ~=~ &a_{11}x^{-1}y + a_{12}y + a_{13}xy\\
&+ a_{21}x^{-1} + a_{22} + a_{13}x\\
&+ a_{31}x^{-1}y^{-1} + a_{32}y^{-1} + a_{33}xy^{-1}
\end{split}
\]
be the equation of $X$ in $(\Cc^*)^2$.

\medskip

The projection $(\Cc^*)^2 \to \Cc^*$ induces a projection $X \cap
(\Cc^*)^2 \to \Cc^*$ given by $(x,y) \mapsto y$. Because $f=0$ is
the same as $xf=0$ on $(\Cc^*)^2$, and
\[\begin{split}
xf ~=~ &(a_{13}y + a_{23} + a_{33}y^{-1})x^2\\
&+ (a_{12}y + a_{22} + a_{32}y^{-1})x\\
&+ (a_{11}y + a_{21} + a_{31}y^{-1}),
\end{split}
\] this projection is a degree $2$ morphism which is ramified along
the discriminant\[(a_{12}y + a_{22} + a_{32}y^{-1})^2 - 4(a_{11}y +
a_{21} + a_{31}y^{-1})(a_{13}y + a_{23} + a_{33}y^{-1})=0.\]

\medskip

Next, we will construct the double mirror $X'$ associated to
$\deg^\vee = \frac{1}{2}(s_1 + s_2)$.

\medskip

Because $\overline{N} = \Zz^2 / \Zz^2 \cap(\Rr s_1 + \Rr s_2)  \cong
\Zz$, and $\Theta  = \overline{K_{(1)}^{\vee}} = \Conv(-1,1)$, the
toric stack $\Pp_\Theta$ is actually the
smooth toric variety $\Cc\Pp^1$. One can find $T_{i,j}$ from
$K_{(1)}$ by pairing with $s_i$. Specifically,
\[\begin{split}
&T_{1,1}\cap M=\{(1,-1,1),(1,-1,0),(1.-1,-1)\},\\
&T_{1,2}\cap M= T_{2,1}\cap M = \{(1,0,1),(1,0,0),(1,0,-1)\},\\
&T_{2,2} \cap M= \{(1,1,1),(1,1,0),(1,1,-1)\}.
\end{split}
\] Hence a global section $g$ in $H^0(\Cc\Pp^1, \Oo(-2K_{\Cc\Pp^1}))$
is \[g=x^{-2} \cdot \det\left(
\begin{array}{cc}
 g_{11}& g_{12}  \\
g_{21}& g_{22}
 \end{array} \right)
\] where
\[\begin{split}
&g_{11} = a_{11}xy^{-1}z + a_{21}xy^{-1} +
a_{31}xy^{-1}z^{-1}\\
&g_{12} = \omega_{21}=\frac{1}{2}(a_{12}xz + a_{22}x
+a_{32}xz^{-1})\\
&g_{22} = a_{13}xyz + a_{23}xy + a_{33}xyz^{-1}.
\end{split}
\] It is straightforward to compute that \[g = (a_{11}z + a_{21} +
a_{31}z^{-1})(a_{13}z + a_{23} + a_{33}z^{-1}) - \frac{1}{4}(a_{12}z
+ a_{22} +a_{32}z^{-1}).\] Hence, the ramification loci (i.e.
$g=0$) on $\Cc\Pp^{1}$ is exactly the same as those of $X$ on
$\Cc^* \subseteq \Cc\Pp^1$. This shows that $X,X'$ are isomorphic
elliptic curves. In particular $D^b(X) \cong D^b(X')$.

\medskip
Of course,  this isomorphism is not unexpected and in fact follows
from the results of our paper as follows. As a consequence of
Theorem \ref{theorem: derived equivalence}, the double mirrors
should have equivalent derived categories. It is well known that derived equivalent smooth
projective curves are isomorphic \cite[Corollary 5.46]{Huy06}.
However, it is worthwhile to see this isomorphism explicitly.

\begin{remark}
One can view this example as the $k=2$ case of Kuznetsov's
construction for complete intersections
of two quadrics in $\Cc\Pp^3$, where we fix one of the quadrics and
view it as a toric variety $\Cc\Pp^1\times \Cc\Pp^1$.
\end{remark}

\subsection{An example of double mirrors without a central fan.}
Let us consider the reflexive Gorenstein cone $K^\vee$ in $\Rr^3$
with the usual lattice $\Zz^3$
generated by
$$
(-1,-1,1),(2,-1,1),(-1,2,1).
$$
The dual cone $K$ is generated by
$$
v_1=(1,0,1),v_2=(0,1,1),v_3=(-1,-1,1)
$$
in $\Zz^3$.
The degree elements are given by
$$
\deg=(0,0,1), ~\deg^\vee=(0,0,1).
$$
The coefficient function $c:K_{(1)}\to \Cc$ is determined by four values
$$
c_1=c(v_1),~c_2=c(v_2),~c_3=c(v_3),~c_0=c(\deg).
$$
The corresponding hypersurface is the elliptic curve whose open part (in the big torus)
is given by the equation
\begin{equation}\label{ell1}
c_1x + c_2y + c_3x^{-1}y^{-1}+c_0 = 0.
\end{equation}
We can think of these curves as the mirrors to the usual elliptic curves in $\Cc\Pp^2$.

\medskip
Consider now the decomposition
$$
\deg^\vee = \frac 12 (0,-1,1) + \frac 12(0,1,1).
$$
On the one hand, the points $s_1=(0,-1,1)$ and $s_2=(0,1,1)$ are in $K^\vee_{(1)}$. On the other
hand, there is no simplicial fan in $K^\vee$ which has all maximum dimensional cones that
contain both of these points. As such, the construction of Section \ref{defcliff} does not apply.

\medskip
Nevertheless, we may consider the construction of Section \ref{sectioncomb}.
We sort the points of $K_{(1)}$ according to their pairings with $s_1$ and $s_2$
and take their convex hulls to get polytopes $T_{ij}$. We get
$$
T_{1,1}=  \{v_3\},~
T_{1,2}=T_{2,1} = {\rm Conv} \{v_1,\deg\} ,~
T_{2,2} =\{v_2\}.
$$
The polytope $S=T$ is then given by
$$
\frac 12 {\rm Conv}\Big( v_3+v_2, 2v_1, v_1+\deg, 2\deg\Big)-\deg
={\rm Conv}\Big((-\frac 12,0,0), (1,0,0)\Big).
$$
The corresponding Calabi-Yau double cover of the toric variety $\Cc\Pp^1$
that corresponds to $2T$ is  given by the determinant of the symmetric  matrix
$$
\det
\left(
\begin{array}{cc}
c_3 x^{-1}y^{-1} & \frac 12 (c_1 x + c_0) \\
 \frac 12 (c_1 x + c_0)& c_2 y
\end{array}
\right)
$$
which is, up to an invertible element, equal to
$$
-4c_2c_3x^{-1} + (c_1x+c_0)^2.
$$
It remains to observe that the elliptic curve
$$
z^2= - 4c_2c_3x^{-1} + (c_1x+c_0)^2
$$
is isomorphic to the one given in \eqref{ell1} under a change of
variables $z=2c_2 y  +c_1x +c_0$.

\begin{remark}
It appears that in $k=1$ case the commutative variety that underlies
the Clifford double mirror is always birational to the original
toric hypersurface. Indeed, the set $K_{(1)}$ has
``width two" and one can look at the corresponding double cover
structure on the hypersurface. Moreover, in $k=1$ (or more generally
$r=1$ case) the even part of Clifford algebra of a two-dimensional
space is commutative, so one simply gets a double cover.
\end{remark}

\subsection{An example without flatness assumption.}\label{noflatness}
In this subsection we discuss a case where the flatness assumption of Remark \ref{flatornot}
does not hold.

\medskip
The variety in question is the Clifford double mirror of the $n$-dimensional
Calabi-Yau hypersurface $H$ of bidegree $(2,n+1)$ in $\Cc\Pp^1\times \Cc\Pp^n$ for $n\geq 3$.
The corresponding reflexive Gorenstein cone $K^\vee$ in the lattice
$$\Zz^2\oplus (\bigoplus_{i=1}^{n+1} \Zz e_i/\Zz \sum_{i=1}^{n+1}e_i)
$$
is generated by
$$
(1,-1;\mathbf{0}),~(1,1;\mathbf{0}),~(1,0;e_i),~ i=1,\ldots,n+1.
$$
The degree element is given by
$\deg^\vee=(1,0;\mathbf{0})$ and is the only
non-vertex lattice point in $K^\vee_{(1)}$. The dual cone $K$ in the
lattice
$$
\Zz^2\oplus \{\sum_{i=1}^{n+1}a_i e_i^\vee, \sum_i a_i=0\}
$$
is generated by
$$(1,-1;(n+1)e_i^\vee-\sum_{j=1}^{n+1}e_j^\vee),~(1,1;(n+1)e_i^\vee-\sum_{j=1}^{n+1}e_j^\vee)
$$
for $i=1,\ldots, n+1$. The degree element is $\deg ={(1,0;\mathbf{0})}$ so the index of the pair of
cones is $k=1$. The lattice points in $K_{(1)}$
$$
(1, j; \sum_{l=1}^{n+1} a_l e_l^\vee
-\sum_{l=1}^{n+1}e_l^\vee),~-1\leq j\leq 1, a_l\geq 0,
\sum_{l=1}^{n+1}a_l= n+1
$$
correspond to monomials $u_1^{j+1}u_2^{1-j} \prod_{l=1}^{n+1} v_l^{a_l}$
for the homogeneous coordinates $(u_1:u_2,v_1:\cdots:v_{n+1})$ on $\Cc\Pp^1\times\Cc\Pp^n$.
The coefficient function $c:K_{(1)}\to \Cc$ encodes the coefficients of the defining equation
of $H$.

\medskip
The Clifford limit corresponds to the decomposition
$$
\deg^\vee = \frac 12 (1,-1;\mathbf{0}) + \frac 12 (1,1;\mathbf{0}).
$$
There is a fan $\Sigma$ which satisfies \eqref{central2} whose
maximum dimensional cones are spanned by $(1,-1;\mathbf 0), (1,1;\mathbf{0})$ and all but
one of $(1,0;e_i)$. The construction of $(\Ss,\Bb_0)$ shows that
$\Ss=\Cc\Pp^n$.

\medskip
To describe $\Bb_0$ observe that $H\to \Cc\Pp^n$ has fibers that are not two disjoint points
over the locus which is the determinant of a symmetric $2\times 2$
matrix of degree $n+1$ forms on $\Cc\Pp^n$. This is a hypersurface $D\subset\Cc\Pp^n$
of degree $(2n+2)$ which is singular at the locus of codimension $3$ in $\Cc\Pp^n$.
The sheaf of algebras
$\Bb_0$ is in this case commutative and is a pushforward of the structure sheaf of the
double cover $H_1$ of $\Cc\Pp^n$ ramified at $D$.
While $H$ is a small resolution of this cover, the map $H\to H_1$ is not an isomorphism
under our assumption $n\geq 3$.

\medskip
In this case, the Clifford double mirror construction produces the
derived category of coherent sheaves on $H_1$. It is not equivalent
to that of $H$ because of the singularities of $H_1$. The reason for
this failure is rooted in the absence of flatness condition of
Remark \ref{flatornot}. It fails precisely over the locus where the
$2\times 2$ matrix is identically zero, which is a complete
intersection of three hypersurfaces of degree $n+1$ in $\Cc\Pp^n$.
The Kuznetsov's theorem is thus inapplicable, so the argument of
Section \ref{sectionderived} falls apart.

\begin{remark}
We believe that even without the flatness or the central fan assumption, there is some
kind of Calabi-Yau geometry associated to the decomposition of $\deg^\vee$ with
coefficients $1$ and $\frac 12$. However, its precise nature appears more complicated.
This is another reason why the primary object of interest is the derived category
$D_B(K,c;\Sigma)$ for a regular simplicial fan $\Sigma$ in $K^\vee$.
\end{remark}

\section{Concluding remarks and open questions.}\label{sectionrem}
There are non-toric double mirror examples which this paper does not address. For example the Pfaffian-Grassmannian double mirrors in \cite{Rod00}, whose derived equivalence is established in \cite{BC09}; the Hosono-Takagi's construction of Reye congruence \cite{HT14}, whose derived equivalence is established in \cite{HT13}. Note also the construction of Bak and Schnell \cite{Bak,Schnell}
related to the Gross-Popescu Calabi-Yau varieties \cite{GP}.
Besides, our method may not be able to cover some examples on derived equivalence between elliptic fibration and its relative Jacobian, see \cite{Cal00a} for discussions on this direction.

\begin{remark}\label{toextend}
Much of the string-theoretic machinery that exists for commutative varieties and even for
DM stacks does not exist in the general setting of noncommutative varieties. This paper
is an indication that it should be possible to define a class of such mildly noncommutative
smooth varieties, that would include smooth DM stacks and to try to extend the following
definitions to this class.
\begin{itemize}
\item
Non-linear sigma models with noncommutative targets.
\item
Gromov-Witten and Donaldson-Thomas invariants.
\item
String-theoretic Hodge numbers (to generalize orbifold Hodge numbers)
\item
Chiral de Rham complexes.
\end{itemize}
\end{remark}

\begin{remark}
Perhaps the most important takeaway from our paper could be the definition of the derived
category associated to reflexive Gorenstein cones. It clearly indicates the need for
better understanding of the theories with potentials. One can pose and try to answer
the same questions as in Remark \ref{toextend} in this setting.
\end{remark}

\begin{remark}
It would be interesting to see if there is a Berglund-H\"ubsch-Krawitz analog of our construction.
\end{remark}

\begin{remark}
We believe that the double mirror construction for Enriques surfaces
in Section \ref{222tau} needs to be investigated further along the lines of Remark \ref{log}.
\end{remark}

\end{document}